\documentclass[11pt]{article}
\usepackage{amsmath}
\numberwithin{equation}{section}
\usepackage{amssymb}
\usepackage{mathtools}

\makeatletter
\DeclareRobustCommand\widecheck[1]{{\mathpalette\@widecheck{#1}}}
\def\@widecheck#1#2{%
    \setbox\z@\hbox{\m@th$#1#2$}%
    \setbox\tw@\hbox{\m@th$#1%
       \widehat{%
          \vrule\@width\z@\@height\ht\z@
          \vrule\@height\z@\@width\wd\z@}$}%
    \dp\tw@-\ht\z@
    \@tempdima\ht\z@ \advance\@tempdima2\ht\tw@ \divide\@tempdima\thr@@
    \setbox\tw@\hbox{%
       \raise\@tempdima\hbox{\scalebox{1}[-1]{\lower\@tempdima\box
\tw@}}}%
    {\ooalign{\box\tw@ \cr \box\z@}}}
\makeatother

\usepackage[backend=biber,maxnames=10,backref=true,hyperref=true,safeinputenc]{biblatex}
\DeclareSourcemap{
  \maps[datatype=bibtex]{
    \map{
      \step[fieldset=note, null]
    }
  }
}

\usepackage{graphicx}
\graphicspath{{images/}}
\usepackage{epstopdf}

\usepackage{algorithm2e}

\newcommand{\ourtitle}{Reconstruction formulas for Photoacoustic Imaging in Attenuating Media}

\title{\ourtitle}
\author{Otmar Scherzer$^{1,2}$\\{\footnotesize\href{mailto:otmar.scherzer@univie.ac.at}{otmar.scherzer@univie.ac.at}}
\and Cong Shi$^1$\\{\footnotesize\href{mailto:cong.shi@univie.ac.at}{cong.shi@univie.ac.at}}}
\usepackage[pdftex,colorlinks=true,linkcolor=blue,citecolor=green,urlcolor=blue,bookmarks=true,bookmarksnumbered=true]{hyperref}
\makeatletter
\AtBeginDocument{
	\hypersetup
	{
	    pdfauthor={Otmar Scherzer, Cong Shi},
	    pdfsubject={Attenuation in Photoacoustic Imaging},
	    pdftitle={\ourtitle},
	    pdfkeywords={Photoacoustic Imaging, Attenuation, Reconstruction formula, Universal backprojection}
	}
}
\makeatother

\makeatletter
\let\oldtheequation\theequation
\renewcommand\tagform@[1]{\maketag@@@{\ignorespaces#1\unskip\@@italiccorr}} 
\renewcommand\theequation{(\oldtheequation)}
\makeatother
\def\equationautorefname~{}

\usepackage[hyperref,amsmath,thmmarks]{ntheorem}
\usepackage{aliascnt}

\newtheorem{lemma}{Lemma}[section]

\newaliascnt{proposition}{lemma}

\aliascntresetthe{proposition}

\newaliascnt{corollary}{lemma}

\aliascntresetthe{corollary}

\newaliascnt{theorem}{lemma}
\newtheorem{theorem}[theorem]{Theorem}
\aliascntresetthe{theorem}

\theorembodyfont{\normalfont}
\newaliascnt{definition}{lemma}
\newtheorem{definition}[definition]{Definition}
\aliascntresetthe{definition}

\newaliascnt{assumption}{lemma}

\aliascntresetthe{assumption}

\newaliascnt{example}{lemma}

\aliascntresetthe{example}

\theoremstyle{nonumberplain}
\theoremseparator{:}
\theoremheaderfont{\normalfont\itshape}

\newtheorem{remark}{Remark}

\theoremsymbol{\ensuremath{\square}}
\newtheorem{proof}{Proof}

\usepackage{dsfont}
\newcommand{\N}{\mathds{N}}

\newcommand{\R}{\mathds{R}}
\newcommand{\C}{\mathds{C}}
\renewcommand{\H}{\mathds{H}}

\let\RE\Re
\let\Re=\undefined
\DeclareMathOperator{\Re}{\RE e}
\let\IM\Im
\let\Im=\undefined
\DeclareMathOperator{\Im}{\IM m}
\DeclareMathOperator{\supp}{supp}

\newcommand{\vx}{\mathbf{x}}

\newcommand{\vxi}{\boldsymbol{\xi}}

\renewcommand{\i}{\mathrm i}
\renewcommand{\d}{\,\mathrm d}

\newcommand{\abs}[1]{\left| #1 \right|}

\newcommand{\set}[1]{\left\{ #1\right\}}

\newcommand{\Ff}[1]{\mathcal{F}[#1]}
\newcommand{\IFf}[1]{\mathcal{F}^{-1}[#1]}

\newcommand{\pseudo}{{\mathcal A}_\kappa}

\renewcommand{\S}{\mathcal{S}}
\newcommand{\temp}[1]{\S'(\R \times \R^{#1})}
\newcommand{\schwarz}[1]{\S(\R \times \R^{#1})}
\newcommand{\opint}[1]{{\mathcal I \left[ #1 \right] }}
\newcommand{\pa}{p^a}
\newcommand{\qa}{q^a}
\newcommand{\Bp}[1]{\mathcal{W}^{-1}\left[ #1 \right]}
\newcommand{\AttOp}[1]{\mathcal{B} \left[ #1 \right]}
\newcommand{\AttOpInv}[1]{\mathcal{B}^{-1} \left[ #1 \right]}
\newcommand{\kinf}{k_\infty}
\newcommand{\Id}{\mathcal{I}d}
\newcommand{\Top}{\mathcal{T}}
\newcommand{\intb}[1]{\int\limits_{#1=-\infty}^\infty \!\!\!}

\usepackage{comment}

\begin{document}

\maketitle
\hspace*{1em}
\parbox[t]{0.49\textwidth}{\footnotesize
\hspace*{-1ex}$^1$Computational Science Center\\
University of Vienna\\
Oskar-Morgenstern-Platz 1\\
A-1090 Vienna, Austria}
\parbox[t]{0.4\textwidth}{\footnotesize
\hspace*{-1ex}$^2$Johann Radon Institute for Computational and Applied Mathematics \\
\hspace*{1em}(RICAM)\\
Altenbergerstra{\ss}e 69\\
A-4040 Linz, Austria}
\vspace*{2em}
\begin{abstract}
 In this paper we study the problem of photoacoustic inversion in a \emph{weakly}
 attenuating medium. We present explicit reconstruction formulas in such media
 and show that the inversion based on such formulas is moderately ill--posed.
 Moreover, we present a numerical algorithm for imaging and demonstrate in
 numerical experiments the feasibility of this approach.
\end{abstract}

\section{Introduction}
When a probe is excited by a short electromagnetic (EM) pulse, it absorbs part of the EM-energy, and expands as a reaction,
which in turn
produces a pressure wave. In \emph{photoacoustic experiments}, using measurements of the pressure wave, the ability of the
medium to transfer absorbed EM-energy into pressure waves is visualized and used for diagnostic purposes.
Common visualizations, see \cite{Wan09}, are based on the assumptions that the specimen can be \emph{uniformly
illuminated}, is acoustically \emph{non-attenuating}, and that the \emph{sound-speed} and \emph{compressibility} are constant.

In mathematical terms, the photoacoustic imaging problem consists in calculating the compactly supported
\emph{absorption density function} $h:\R^3 \to \R$, appearing as a source term in the wave equation
\begin{equation} \label{eqWaveEquation}
\begin{aligned}
\partial_{tt}p(t,\vx)-\Delta p(t, \vx)&=\delta'(t)h(\vx),\quad &&t\in\R,\;\vx \in \R^3, \\
p(t,\vx)&=0,\quad&&t<0,\;\vx\in\R^3,
\end{aligned}
\end{equation}
from some measurements over time of the pressure $p$ on a two-dimensional manifold $\Gamma$ outside of the
specimen, that is outside of the support of the absorption density function. This problem has been studied extensively
in the literature (see e.g. \cite{KucKun08,WanAna11,Kuc14}, to mention just a few survey articles).

Biological tissue has a non-vanishing viscosity, thus there is thermal consumption of energy. These effects can be described
mathematically by \emph{attenuation}. Common models of such are the \emph{thermo-viscous} model \cite{KinFreCopSan00},
its modification \cite{KowSchBon11}, Szabo's power law \cite{Sza94,Sza95} and a causal modification \cite{KowSch12},
Hanyga \&  Seredy{'n}ska \cite{HanSer03}, Sushilov \& Cobbold \cite{SusCob04},
and the Nachman--Smith--Waag model \cite{NacSmiWaa90}.
Photoacoustic imaging in attenuating medium then consists in computing the absorption density function $h$ from measurements
of the attenuated pressure $\pa$ on a surface containing the object of interest. The attenuated pressure equation reads as
follows
\begin{equation}\label{eq:WEA3}
\begin{aligned}
\pseudo[\pa](t,\vx)-\Delta \pa(t,\vx)&=\delta'(t)h(\vx),\quad &&t\in\R,\;\vx\in\R^3, \\
\pa(t,\vx)&=0,\quad&&t<0,\;\vx\in\R^3,
\end{aligned}
\end{equation}
where $\pseudo$ is the pseudo-differential operator defined in frequency domain (see \autoref{eqAttenuationOp}).
The formal difference between \autoref{eqWaveEquation} and \autoref{eq:WEA3} is that the second time derivative
operator $\partial_{tt}$ is replaced by a pseudo-differential operator $\pseudo$.

The literature on Photoacoustics in attenuating media concentrates on time-reversal and attenuation compensation
based on power laws: We mention the $k$-wave toolbox implementation and the according papers
\cite{CoxTre10,TreZhaCox10}, \cite{BurGruHalNusPal07,HuaNieSchWanAna12}.
In \cite{ElbSchShi16_report} several attenuation laws from the literature have been cataloged into
two classes, namely \emph{weak} and \emph{strong} attenuation laws.
Power laws lead, in general, to severely ill--posed photoacoustic imaging problem, while mathematically
sophisticatedly derived models, like the Nachman-Smith-Waag model
\cite{NacSmiWaa90}, lead to moderately ill--posed problems.

The paper is based on the premise that Photoacoustics is moderately ill--posed, and we therefore concentrate
on photoacoustic inversion in weakly attenuating media, which have not been part of extensive analytical and
numerical studies in the literature. Another goal of this work is to derive \emph{explicit} reconstruction
formulas for the absorption density function $h$ in attenuating media. Previously there have been derived
asymptotical expansions in the case of small absorbers \cite{AmmBreGarWah12,KalSch13}.

\subsection*{Notation}
We use the following notations:
\begin{itemize}
 \item For $s=0,1,2,\ldots$ we denote by $\schwarz{s}$ the Schwartz-space of
       complex valued functions and its dual space, the space of \emph{tempered distribution}, is denoted by
       $\temp{s}$.

       We abbreviate $\S = \schwarz{3}$ and $\S'=\temp{3}$.
 \item For $\phi \in \S(\R)$ we define the Fourier-transform by
       \begin{equation*}
        \hat{\phi}(\omega) =
        \frac{1}{\sqrt{2 \pi}}\intb{t} e^{\i \omega t} \phi(t)\d t,
       \end{equation*}
       and the one-dimensional inverse Fourier-transform is given by
       \begin{equation*}
        \check{\varphi}(t) =
        \frac{1}{\sqrt{2 \pi}} \intb{\omega} e^{-\i \omega t} \varphi(\omega)\d \omega.
       \end{equation*}
 \item Let $\varphi \in \S(\R)$ and $\psi \in \S (\R^3)$.
       The Fourier-transform $\Ff{\cdot}: \S' \to \S'$ on the space of tempered distributions is defined by
       \begin{equation}\label{eq:op_tempered_fourier}
         \left<\Ff{u},\varphi \otimes \psi\right>_{\S',\S} =
         \left<u, \check{\varphi} \otimes \psi \right>_{\S',\S}.
       \end{equation}
       Note that for functions $u \in \S$ we have
       \begin{equation*}
         \left<\Ff{u},\varphi\otimes\psi\right>_{\S',\S} =
         \int\limits_{\R \times \R^3} u(t, \vx) \overline{\check{\varphi}(t) \psi(\vx)} \d t \d \vx \;.
       \end{equation*}
       We are identifying distributions and functions and we are writing in the following for all $u \in \S'$
       \begin{equation*}
       \begin{aligned}
         \Ff{u}(\omega,\vx) &= \frac{1}{\sqrt{2\pi}} \intb{t} e^{\i \omega t} u(t,\vx)\d t \,\text{ and }\\
         \IFf{y}(t,\vx) &= \frac{1}{\sqrt{2\pi}} \intb{\omega} e^{-\i \omega t} y (\omega,\vx)\d \omega.
       \end{aligned}
       \end{equation*}
 \item We define the attenuation operator $\pseudo[\cdot]: \S' \rightarrow \S'$ by
       \begin{equation}\label{eqAttenuationOp}
        \left< \pseudo[u], \phi\otimes\psi\right>_{\S',\S} = -
        \left< u, \widecheck{\overline{\kappa^2} \hat{\phi}} \otimes\psi\right>_{\S',\S}.
       \end{equation}
       This means that if $u \in \schwarz{3}$ then
       \begin{equation}
       \label{eq:pseudo}
         \pseudo[u](t,\vx) = -\IFf{\kappa^2 \Ff{u}}(t,\vx),\quad \omega \in \R,\;\vx \in \R^3.
       \end{equation}
 \item $\opint{\cdot}$ denotes the \emph{time integration operator on the space of tempered distributions} and is given by
       \begin{equation}\label{eq:int_tempered_distribution}
         \left<\opint{u},\phi\otimes\psi\right>_{\S',\S} =
         - \left<u,\phi' \otimes \psi \right>_{\S',\S},
       \end{equation}
       and we write formally for $u \in \S'$
       \begin{equation} \label{eq:intop}
         u \to \opint{u}(t,\vx) = \int\limits_{-\infty}^t u(\tau,\vx)\,d\tau.
       \end{equation}
 \item \autoref{eq:WEA3} and \autoref{eqWaveEquation} (here $\kappa(\omega)=\omega^2$)
       are understood in a distributional sense, which means that
       for all $\phi \in \S(\R)$ and $\psi \in \S (\R^3)$
       \begin{equation}\label{eq:paWeakForm}
       \begin{aligned}
        ~ & \left<\pa,\widecheck{\overline{\kappa^2} \hat{\phi}} \otimes\psi \right>_{\S',\S} +
        \left<\pa, \phi \otimes \Delta_{\vx} \psi \right>_{\S',\S} \\
        = & \phi'(0)
        \left<h,\psi\right>_{\S'(\R^3),\S(\R^3)}\;.
       \end{aligned}
       \end{equation}
\end{itemize}

\section{Attenuation} \label{sec:attenuation}
Attenuation describes the physical phenomenon that certain frequency components of acoustic waves are
attenuated more rapidly over time.
Mathematically this is encoded in the function $\kappa$ defining the pseudo-differential operator $\pseudo$.
A physically and mathematically meaningful $\kappa$ has to satisfy the following properties
(see \cite{ElbSchShi16_report}):
\begin{definition}\label{deAttCoeff}
 We call a non-zero function $\kappa\in C^\infty(\R;\overline\H)$, where $\H =\{z\in\C : \Im z>0\}$
 denotes the upper half complex plane and $\overline{\H}$ its closure in $\C$, an \emph{attenuation coefficient} if
 \begin{enumerate}
  \item\label{enAttCoeffPolBdd}
       all the derivatives of $\kappa$ are polynomially bounded. That is, for every $\ell\in\N_0$ there exist
       constants $\kappa_1>0$ and $N\in\N$ such that
       \begin{equation}\label{eqAttCoeffPolBdd}
        |\kappa^{(\ell)}(\omega)| \le \kappa_1(1+|\omega|)^N,
       \end{equation}
  \item\label{enAttCoeffHol}
       There exists a holomorphic continuation $\tilde\kappa:\overline{\H}\to\overline{\H}$ of $\kappa$ on the upper
       half plane, that is, $\tilde\kappa\in C(\overline{\H};\overline{\H})$ with $\tilde\kappa|_{\R}=\kappa$ and
       $\tilde\kappa:\H\to\overline{\H}$ is holomorphic, and there exists constants $\tilde\kappa_1>0$ and
       $\tilde N\in\N$ such that
       \[ |\tilde\kappa(z)|\le \tilde\kappa_1(1+|z|)^{\tilde N}\quad\text{for all}\quad z\in\overline{\H}.\]
  \item\label{enAttCoeffSymm}
        $\kappa$ is symmetric: That is, $\kappa(-\omega)=-\overline{\kappa(\omega)}$ for all $\omega\in\R$.
  \item There exists some constant $c>0$ such that the holomorphic extension $\tilde\kappa$ of the attenuation
        coefficient $\kappa$ satisfies
        \begin{equation*}
         \Im(\tilde\kappa(z)-\tfrac zc)\ge0\quad\text{for every } z\in\overline\H.
        \end{equation*}
 \end{enumerate}
\end{definition}

\begin{remark}
The four conditions in \autoref{deAttCoeff} on $\kappa$ encode the following physical properties
of the attenuated wave equation (see \cite{ElbSchShi16_report}):
\begin{description}
 \item{The condition \autoref{eqAttCoeffPolBdd}} in \autoref{deAttCoeff} ensures that the product $\kappa^2u$ of $\kappa^2$
       with an arbitrary tempered distribution $u\in\S'$ is again in $\S'$ and therefore, the operator $\pseudo: \S' \to \S'$ is
       well-defined on the space of tempered distributions.
 \item{The second condition} guarantees that the solution of the attenuated wave equation is causal.
 \item{Condition three} ensures that real valued distributions (such as the pressure) are mapped to real valued distributions.
       That is $\pa$ is real valued if the absorption density $h$ is real.
 \item{The forth condition} guarantees that the solution $\pa \in\mathcal S'(\R\times\R^3)$ of the equation \autoref{eq:WEA3}
       \emph{propagates with finite speed} $c>0$. That is
       \[ \supp \pa\subset\{(t,x)\in\R\times\R^3: |x|\le ct+R\} \] whenever $\supp h\subset B_R(0)$.
\end{description}
\end{remark}

In the literature there have been documented two classes of attenuation models:
\begin{definition}\label{deAttenuation}
We call an attenuation coefficient $\kappa\in C^\infty(\R;\overline\H)$ (see \autoref{deAttCoeff})
\begin{itemize}
\item a \emph{weak attenuation} coefficient if there exists a constant $0 \leq \kinf \in \R$ and a bounded function
      $k_*\in C^\infty(\R;\C)\cap L^2(\R;\C)$ such that
      \begin{equation}\label{eq:weak_attenuation}
       \kappa(\omega) = \omega +\i \kinf +k_*(\omega) \text{ for all } \omega\in\R.
      \end{equation}
      In particular, $\kappa$ is \emph{constantly attenuating}, if $\kappa$ is a weak attenuation coefficient with
      $k_* \equiv 0$.
      That is, there exists a constant $\kinf \ge 0$ such that
      \begin{equation}\label{eq:constant_attenuation}
       \kappa(\omega) = \omega +\i \kinf \text{ for all } \omega\in\R.
      \end{equation}
\item $\kappa$ is called \emph{strong attenuation coefficient} if there exist constants $\kappa_0 > 0$, $\beta > 0$, and $\omega_0 > 0$
      \begin{equation}\label{eq:strong_attenuation}
       \Im \kappa(\omega) \geq \kappa_0 \abs{\omega}^\beta \text{ for all } \omega \in \R \text{ with } \abs{\omega} \geq \omega_0.
      \end{equation}
\end{itemize}
\end{definition}

For such attenuation coefficients we proved in \cite{ElbSchShi16_report} well-posedness of the attenuated wave equation:

\begin{lemma}\label{le:wellposedness}
 Let $\kappa$ be an attenuation coefficient. Then the solution $\pa$ of the equation \autoref{eq:WEA3}
 exists and is a real-valued tempered distribution in $\R\times\R^3$.

 Moreover, $\qa :=\opint{\pa}$ is a tempered distribution and satisfies the equation
 \begin{equation}\label{eq:int_wave_eq}
 \begin{aligned}
 \pseudo[\qa](t,\vx)-\Delta \qa(t,\vx)&=\delta(t)h(\vx),\quad &&t\in\R,\;\vx\in\R^3, \\
 \qa(t,\vx)&=0,\quad&&t<0,\;\vx\in\R^3,
 \end{aligned}
 \end{equation}
 and in Fourier domain
 \begin{equation}\label{eq:int_Helmholtz}
  \kappa^2(\omega) \Ff{\qa}(\omega,\vx)+\Delta_\vx \Ff{\qa}(\omega, \vx)= -\frac{1}{\sqrt{2\pi}}h(\vx).
 \end{equation}
\end{lemma}
\begin{proof} The first item has been proven in \cite{ElbSchShi16_report}. The second item is an immediate consequence of the
 definition of a tempered distribution.
\end{proof}

\begin{remark}
\autoref{eq:int_wave_eq} has to be understood in a distributional sense:  That is $\qa \in \S'$ and satisfies
for every $\phi\in \S(\R)$ and $\psi\in \S(\R^3)$ the equation
\begin{equation}
\label{eq:qa_equation_weak}
 \left<\pseudo[q^a],\phi\otimes\psi\right>_{\S',\S}-\left<q^a,\phi\otimes\Delta_{\vx}\psi\right>_{\S',\S} = \phi(0)\left<h,\psi\right>_{\S'(\R^3),\S(\R^3)}.
\end{equation}
If $\kappa(\omega)=\omega$ (that is the case of the standard wave equation) $q=\qa$ solves the following equation
in a distributional sense
\begin{equation} \label{eqWaveEquationQ}
\begin{aligned}
\partial_{tt}q(t,\vx)-\Delta_\vx q(t, \vx)&=\delta(t)h(\vx),\quad &&t\in\R,\;\vx \in \R^3, \\
q(t,\vx)&=0,\quad&&t<0,\;\vx\in\R^3,
\end{aligned}
\end{equation}
and its Fourier-transform $\Ff{q}$ satisfies the Helmholtz equation
\begin{equation}\label{eqHelmHoltzQ}
\omega^2 \Ff{q}(\omega,\vx)+\Delta_\vx \Ff{q}(\omega, \vx)= -\frac{1}{\sqrt{2\pi}}h(\vx), \quad \omega\in\R,\;\vx \in \R^3.
\end{equation}
Again $q \in \temp{3}$ and satisfies \autoref{eqWaveEquationQ} in a distributional sense:
\begin{equation}
\label{eq:q_equation_weak}
\left<q,\partial_{tt}\phi\otimes\psi\right>_{\S',\S}-\left<q,\phi\otimes\Delta_{\vx}\psi\right>_{\S',\S} = \phi(0)\left<h,\psi\right>_{\S'(\R^3),\S(\R^3)}.
\end{equation}
The solution of \autoref{eqWaveEquationQ} can also be written as the solution of the initial value problem:
\begin{equation} \label{eqWaveEquationQ1}
\begin{aligned}
\partial_{tt}q(t,\vx)-\Delta_\vx q(t, \vx) &= 0,\quad && t > 0,\;\vx \in \R^3, \\
q(0,\vx)&=0,\quad&& \vx\in\R^3,\\
\partial_t q(0,\vx)&=h(\vx),\quad&& \vx\in\R^3.
\end{aligned}
\end{equation}
\end{remark}

In the following we derive a functional relation between $q$ and $\qa$, which is the basis of
analytical back-projection formulas in attenuating media.
\begin{theorem}
\label{eq:nu}
Let $\phi \in \S(\R)$ and define
\begin{equation*}
\nu_+[\phi](\tau):= \frac{1}{\sqrt{2\pi}} \intb{\omega}
                             e^{-\i \overline{\kappa(\omega)} \tau}\hat{\phi}(\omega) \d \omega \text{ for all } \tau \geq 0\;.
\end{equation*}
Then there exists a sequence $(a_m)_{m\geq 1}$ of real numbers satisfying $\sum_{m \geq 1} a_m 2^{m j}=(-1)^j$ and a function
$\vartheta \in C^\infty_0(\R;\R)$ such that $\vartheta(\tau)=1$ when $|\tau|<1$ and $\vartheta(\tau)=0$ when
$|\tau| \geq 2$ such that
the function $\nu: \R \to \C$, defined by
\begin{equation}\label{eq:nu_ext}
\nu(\tau):=\nu[\phi](\tau):=
       \left\{ \begin{array}{lcr}
                \nu_+[\phi](\tau) & \text{ for all } &\tau \geq 0,\\
                \sum_{m=0}^{\infty} a_m \nu_+[\phi](-2^m\tau)\vartheta(-2^m\tau) & \text{ for all } &\tau<0,
               \end{array}\right.
\end{equation}
is an element of the Schwartz space $\S(\R)$.
\end{theorem}
\begin{proof}
\begin{enumerate}
 \item Using that for all $k \in \N_0$
       \begin{equation*}
         \psi_k(\tau,\omega) := \partial_\tau^k e^{-\i \overline{\kappa(\omega)} \tau} =
         (-\i)^k \overline{\kappa(\omega)}^k e^{-\i \overline{\kappa(\omega)} \tau} \text{ for all } \omega,\tau \in \R,
       \end{equation*}
       it follows from \autoref{deAttenuation} that, uniformly in $\tau$, for all $\omega \in \R$
       \begin{equation*}
         \abs{\psi_k(\tau,\omega)\hat{\phi}(\omega)}
         \leq \abs{\kappa(\omega)}^k e^{-\Im \kappa(\omega) \tau}\abs{\hat{\phi}(\omega)} \leq \abs{\kappa(\omega)}^k \abs{\hat{\phi}(\omega)}\;.
       \end{equation*}
       From \autoref{eqAttCoeffPolBdd} and $\hat{\phi} \in \S(\R)$ (in particular $\kappa \in L^\infty(\R;\C)$ and
       $\hat{\phi} \in C(\R;\C)$),
       $\omega \to \abs{\kappa(\omega)}^k \abs{\hat{\phi}(\omega)} \in L^1(\R; \R)$.
       Thus by interchanging integration and differentiation it follows that for $R \to \infty$
       \begin{equation}\label{eq:nu_ext_der}
         d^k \nu_+[\phi] (\tau) = \frac{1}{\sqrt{2\pi}} \intb{\omega}
                              \psi_k(\tau,\omega) \hat{\phi}(\omega)
                              \d \omega \text{ for all } \tau \geq 0
       \end{equation}
       and that these functions are continuous. Thus $\nu_+[\phi] \in C^{\infty}([0,\infty); \C)$.
 \item From \cite{See64} it follows that $\nu[\phi]$ defined in \autoref{eq:nu_ext} is in $C^\infty(\R; \C)$ and
       extends the function $\nu_+[\phi]$ defined on $[0,\infty)$.
 \item We are proving that $\nu_+[\phi]$ and all its derivative are faster decaying than polynomials in $\tau$ for
       $\tau\to +\infty$.

       For this purpose we use the \emph{stationary phase method} summarized in \autoref{th:st_phase}:

       Let $\theta \in C^\infty_0(\R; \R)$ be a mollifier satisfying $\theta(\omega)=1$
              when $|\omega|<1$ and $\theta(\omega)=0$ when $|\omega| \geq 2$.
              For $k \in \N$ and $R > 0$ fixed, we apply the stationary phase method with
              \begin{equation*}
               \omega \to f(\omega)=-\overline{\kappa(\omega)} \text{ and }
               \omega \to g_R(\omega):=\psi_k(\tau,\omega)\theta(\omega/R).
              \end{equation*}
              Below we are verifying the assumptions of the stationary phase method:
              \begin{itemize}
              \item $f \in C^\infty(\R;\C)$ by \autoref{deAttCoeff} and $g_R \in C^\infty_0(\R; \C)$, because it is the
                    product of a $C^\infty(\R;\C)$ function and the compactly supported function $\omega \to \theta(\omega/R)$.
              \item The second property
                    \begin{equation} \label{eq:help}
                     \Im f = \Im \kappa \geq 0\;.
                    \end{equation}
                    is an immediate consequence of the assumption $\kappa \in C^\infty(\R;\overline{\H})$ in
                    \autoref{deAttCoeff}.
              \end{itemize}
       Thus \autoref{eq:stationary_phase} can be applied with the functions $f=-\overline{\kappa}$ and $g_R$, and
       using \autoref{le:hoermander}
       \autoref{eq:klb_derivative} it follows that
       \begin{equation}\label{eq:stationary_phase_1}
        \begin{aligned}
         & \tau^l \left|\intb{\omega} e^{\i \tau f(\omega)}g_R(\omega) \d \omega\right|\\
         \leq &
         C_1\sum_{\alpha=0}^{l} \sup_{\omega\in\R} \left|d^\alpha g_R(\omega)\right| (|f'(\omega)|^2 + \Im f(\omega)) ^{\alpha/2-l}\\
         \leq &
         C_1\sum_{\alpha=0}^{l}\sup_{\omega\in\R}\left|d^\alpha g_R(\omega)\right|C_2^{\alpha/2-l}.
        \end{aligned}
       \end{equation}

       Next, we consider the limit $R \to \infty$.
       Because $\theta \in C_0^\infty(\R; \R)$ and $g \in \S(\R)$ and
       \begin{equation*}
       d^\alpha g_R(\omega)
       = \sum_{\beta=0}^{\alpha} d^{\alpha-\beta} g(\omega)
       \binom{\alpha}{\beta} \frac{1}{R^\beta} d^\beta \theta(\omega),
       \end{equation*}
       it follows that
       \begin{equation*}
       \left|d^\alpha g_R(\omega) - d^\alpha g(\omega)\right|
       \leq \left|
       \sum_{\beta=1}^{\alpha}d^{\alpha-\beta} g(\omega) \binom{\alpha}{\beta}\frac{1}{R^\beta}d^\beta \theta(\omega) \right| = \mathcal{O}(R^{-1}).
       \end{equation*}
       and from \autoref{eq:help} and the assumption $\tau \geq 0$ it follows
       \begin{equation*}
       \begin{aligned}
       ~ & \left|\intb{\omega} e^{\i \tau f(\omega)}g_R(\omega) \d \omega - \intb{\omega}
         e^{\i \tau f(\omega)}g(\omega) \d \omega\right| \\
       \leq & \intb{\omega} e^{-\tau \Im f(\omega)} \left|g_R(\omega)- g(\omega)\right|\d \omega
         \quad \leq  \intb{\omega} \left|g_R(\omega)- g(\omega)\right|\d \omega.
       \end{aligned}
       \end{equation*}
       Using the definition of $g_R$ it follows that
       \begin{equation*}
       \begin{aligned}
         \intb{\omega} \left|g_R(\omega)- g(\omega)\right|\d \omega &\leq \intb{\omega} \left|g(\omega)\right| \left| 1- \theta(\omega/R)\right| \d \omega\\
         &\leq \int\limits_{|\omega|\geq R} \left|g(\omega)\right| \d \omega.
         \end{aligned}
       \end{equation*}
       Since $g\in \S(\R)$ the last integral tends to $0$ for $R \to \infty$ and thus
       \autoref{eq:stationary_phase} holds even for the functions $\omega \to f(\omega)=-\overline{\kappa(\omega)}$ and
       $\omega \to g(\omega) = \psi_k(\tau,\omega)$, although they are not satisfying the assumptions of \autoref{th:st_phase}.
%
%

\end{enumerate}
\end{proof}

Because (according to \autoref{eq:nu}), for every $\phi \in \S(\R)$, $\nu[\phi] \in \S(\R)$, the operator from the following
definition is well-defined.
\begin{definition}
The \emph{attenuation solution operator} $\AttOp{\cdot}: \S' \to \S'$ is defined by
 \begin{equation} \label{eq:mult}
 \begin{aligned}
  \left<\AttOp{\rho}, \phi \otimes \psi\right>_{\S',\S} = &
  \left<\rho, \nu[\phi] \otimes \psi \right>_{\S',\S} \\
  & \text{ for all } \rho \in \S' \text{ and }
  \phi \in \S(\R), \psi \in \S (\R^3)\;.
 \end{aligned}
 \end{equation}
\end{definition}

\begin{remark}
 In a weakly attenuating medium $\kappa(\omega) = \omega +\i \kinf + k_*(\omega)$, and therefore,
 for every $t, \tau \in \R$
      \begin{equation}
      \label{eq:tau}
      \begin{aligned}
       k(t,\tau)
       &:= \IFf{e^{\i(\cdot+\i \kappa _\infty+ k_*(\cdot))\tau}}(t) \\
       &=  e^{-\kinf \tau} \IFf{e^{\i k_*(\cdot)\tau}}(t-\tau)\\
       &=  e^{-\kinf \tau} \IFf{1+ (e^{\i k_*(\cdot)\tau}-1)}(t-\tau) \\
       & = \sqrt{2 \pi}e^{-\kinf \tau} \delta(t-\tau) + e^{-\kinf \tau} \IFf{e^{\i k_*(\cdot)\tau}-1}(t-\tau).
      \end{aligned}
      \end{equation}
      Because there exists a constant $C>0$ such that
      \begin{equation*}
       \abs{e^{\i k_*(\omega)\tau}-1} \leq C \abs{ k_*(\omega)\tau} \text{ for all }
       \omega,\tau \in \R,
      \end{equation*}
      it follows from \autoref{deAttenuation} (stating that $k_* \in L^2(\R; \C)$)
      and Plancharel's identity that
      \begin{equation*}
      \begin{aligned}
              \omega &\to e^{\i k_*(\omega)\tau}-1 \in L^2(\R;\C) \text{ and } \\
              t & \to e^{-\kinf \tau} \IFf{e^{\i k_*(\cdot)\tau}-1}(t-\tau) \in L^2(\R;\C) \text{ for all } \tau \in \R\;.
      \end{aligned}
      \end{equation*}
      Now, assume that $\rho \in C^0(\R \times \R^3;\C) \cap \S'$ with support in $[0,\infty) \times \R^3$,
      $\phi \in \S(\R)$ and $\psi \in \S(\R^3)$
 \begin{equation*}
 \begin{aligned}
  ~ & \left<\rho, \nu[\phi] \otimes \psi \right>_{\S',\S}\\
  = &
  \int\limits_{\R^3} \intb{\tau} \rho(\tau,\vx) \overline{\nu[\phi](\tau) \psi(\vx)}\,\d \tau \d \vx\\
  = &
  \frac{1}{\sqrt{2\pi}} \int\limits_{\R^3} \overline{\psi(\vx)}
  \intb{\tau} \rho(\tau,\vx) \intb{\omega}
                             e^{\i \kappa(\omega) \tau}\hat{\overline{\phi}}(\omega) \d \omega \,\d \tau \d \vx\;.
 \end{aligned}
 \end{equation*}
 Using Parseval's identity we get
     \begin{equation}\label{eq:kernel}
 \begin{aligned}
  ~ & \left<\rho, \nu[\phi] \otimes \psi \right>_{\S',\S}\\
  = &
  \frac{1}{\sqrt{2\pi}} \int\limits_{\R^3} \intb{\tau} \rho(\tau,\vx) \overline{\psi(\vx)} \intb{t}
                             k(t,\tau) \overline{\phi(t)} \d t \,\d \tau \d \vx\\
  = &
  \int\limits_{\R^3} \intb{t} \overline{\psi(\vx)} \overline{\phi(t)}
  \underbrace{ \left( \frac{1}{\sqrt{2\pi}} \intb{\tau} k(t,\tau) \rho(\tau,\vx)
                               \d \tau \right)}_{= \AttOp{\rho}(t,\vx) } \d t \d \vx.
 \end{aligned}
 \end{equation}
 Then for $\rho \in C^0(\R \times \R^3; \C) \cap \S'$ it follows from \autoref{eq:kernel} and \autoref{eq:tau} that
      \begin{equation}\label{eq:operatorT_1}
      \boxed{
      e^{\kinf t}\AttOp{\rho}(t,\vx) = (\Id + \Top)[\rho](t,\vx)\,,}
      \end{equation}
      where
      \begin{equation}
       \label{eq:top}
       \Top[\rho](t,\vx) =
       \frac{1}{\sqrt{2\pi}} \intb{\tau} e^{\kinf (t-\tau)} \IFf{e^{\i k_*(\omega)\tau}-1}(t-\tau)\rho(\tau,\vx)\,\d \tau\;.
      \end{equation}
\end{remark}

\begin{theorem}\label{thRelationQaQ}
Let
$q = \opint{p}$ and $\qa = \opint{\pa}$, where $p$ and $\pa$ are the solutions of the
equations \autoref{eqWaveEquation} and \autoref{eq:WEA3}, respectively.
Then
\begin{equation}\label{eq:qa_q}
 \qa = \AttOp{q}\;.
\end{equation}
\end{theorem}

\begin{proof}
Let $\phi \in \S(\R)$ and $\psi \in \S(\R^3)$.
Then, from \autoref{eq:qa_q}, the definition of the Fourier-transform \autoref{eq:op_tempered_fourier},
the definition of $\AttOp{\cdot}$ in \autoref{eq:mult}, because $\nu[\phi] \in \S(\R)$ (see \autoref{eq:nu_ext})
and because $q$ solves \autoref{eqWaveEquationQ} it follows that
\begin{equation} \label{eq:ak}
 \begin{aligned}
  ~& \left< \pseudo[\AttOp{q}], \phi \otimes  \psi \right>_{\S',\S} - \left< \AttOp{q}, \phi \otimes \Delta_{\vx}\psi \right>_{\S',\S}\\
  =& - \left< \AttOp{q}, \widecheck{\overline{\kappa^2}\hat{\phi}} \otimes  \psi \right>_{\S',\S}
     - \left< \AttOp{q}, \phi \otimes \Delta_{\vx}\psi \right>_{\S',\S}\\
  =& -\left< q, \nu\left[\widecheck{\overline{\kappa^2}\hat{\phi}}\right] \otimes  \psi \right>_{\S',\S} -
     \left< q, \nu[\phi] \otimes \Delta_\vx \psi \right>_{\S',\S}\\
  =& - \left<q, \nu\left[\widecheck{\overline{\kappa^2}\hat{\phi}}\right] \otimes  \psi \right>_{\S',\S} -
            \left<  q, \partial_{\tau \tau} \nu[\phi] \otimes  \psi \right>_{\S',\S}\\
   &\quad         + \nu[\phi](0)\left< h, \psi \right>_{\S'(\R^3),\S(\R^3)}
  \end{aligned}
 \end{equation}
 We are representing every term on the right hand side:
 \begin{enumerate}
 \item From \autoref{eq:nu} it follows that
       \begin{equation*}
        \nu[\phi](0) = \frac{1}{\sqrt{2\pi}}\intb{\omega}\hat{\phi}(\omega)\d\omega
        = \check{\hat{\phi}}(0)
        = \phi(0)
       \end{equation*}
       and thus
       \begin{equation*}
        \begin{aligned}
        \nu[\phi](0)\left< h, \psi \right>_{\S'(\R^3),\S(\R^3)} = &\phi(0)\left< h, \psi \right>_{\S'(\R^3),\S(\R^3)}.
        \end{aligned}
       \end{equation*}
  \item The first term on the right hand side of \autoref{eq:ak} can be represented as follows:
        \begin{equation*}
          \left<q, \nu\left[\widecheck{\overline{\kappa^2} \hat{\phi}}\right] \otimes  \psi \right>_{\S',\S}
        = \left<q, \frac{1}{\sqrt{2\pi}} \intb{\omega}
           \overline{\kappa^2(\omega)} e^{-\i \overline{\kappa(\omega)} \cdot}
           \hat{\phi}(\omega)\d\omega \otimes  \psi \right>_{\S',\S}\!\!\!\!\!\!.
        \end{equation*}
  \item For the second term we find that
        \begin{equation*}
         \begin{aligned}
          \left< q, \partial_{\tau \tau} \nu[\phi] \otimes  \psi \right>_{\S',\S}
          = & \left< q, \frac{1}{\sqrt{2\pi}}\intb{\omega} \partial_{\tau \tau}
               e^{-\i \overline{\kappa(\omega)} \, \cdot}\; \hat{\phi}(\omega)\d\omega\otimes \psi \right>_{\S',\S}\\
          = & - \left< q, \frac{1}{\sqrt{2\pi}}\intb{\omega} \overline{\kappa^2(\omega)}
               e^{-\i \overline{\kappa(\omega)} \cdot}\hat{\phi}(\omega)\d\omega\otimes \psi \right>_{\S',\S}.
\end{aligned}
\end{equation*}
\end{enumerate}
The sum of the first and second term vanishes and thus from \autoref{eq:ak} it follows that
\begin{equation*}
 \begin{aligned}
  ~& \left< \pseudo[\AttOp{q}], \phi \otimes  \psi \right>_{\S',\S} - \left< \AttOp{q}, \phi \otimes \Delta_{\vx}\psi \right>_{\S',\S}\\
  =& \phi(0)\left< h, \psi \right>_{\S'(\R^3),\S(\R^3)}\,,
  \end{aligned}
 \end{equation*}
which shows that $\AttOp{q}$ solves \autoref{eq:qa_equation_weak}. Since the solution of this equation is unique it follows that
$\qa = \AttOp{q}$.
\end{proof}

\section{Reconstruction formulas}
In this section we provide explicit reconstruction formulas for the absorption density $h$
(the right hand side of \autoref{eq:WEA3}) in {\bf attenuating} media. The basis of these
formulas are exact reconstruction formulas in {\bf non-attenuating} media.

In the case of non-attenuating media the problem of photoacoustic tomography consists in determining the absorption
density $h$ from measurement data of
\begin{equation*}
 m(t,\vxi) := p(t, \vxi) \text{ for all } t>0, \vxi \in \Gamma\,,
\end{equation*}
where $\Gamma$ denotes the measurement surface and $p$ is the solution of \autoref{eqWaveEquation}.

Let $\mathcal{W}$ be the operator which maps $h$ to $p$. The most universal (meaning applicable for a series of
measurement geometries $\Gamma$) formula
for $\Bp{\cdot}$ is due to Xu \& Wang \cite{XuWan05}. Several different formulas of such
kind were presented and analyzed in \cite{KucKun08,Nat12,Kuc12}.
The formula of Xu \& Wang \cite{XuWan05} in $\R^3$ reads as follows:
\begin{equation}\label{eq:UniversalBP-3D}
h(x)=\frac{2}{\Omega_0} \int\limits_{\vxi \in \Gamma} \frac{p(|\vxi-\vx|, \vxi)-|\vxi-\vx| \frac{\partial p}{\partial t}(|\vxi-\vx|, \vxi)}{|\vxi-\vx|^2 } \left( \mathbf{n}_{\vxi}\cdot \frac{\vxi-\vx}{|\vxi-\vx|} \right) \d s(\Gamma)
\end{equation}
where $\Omega_0$ is $2\pi$ for a planar geometry and $4\pi$ for cylindrical and spherical geometries and
$\mathbf{n}_{\vxi}$ is the outer normal vector for the measurement surface $\Gamma$.

From \autoref{thRelationQaQ} we get an explicit reconstruction formula in the case of attenuating media:
\begin{theorem}
Under the assumption that the universal back-projection can be applied in the non-attenuating case, we have
\begin{equation} \label{eq:explicit}
h=\Bp{\partial_t \AttOpInv{(t,\vxi) \to \qa(t,\vxi)}}.
\end{equation}
\end{theorem}

In the following we study the attenuation operator $\AttOp{\cdot}$ and its inverse
in weakly attenuating media.

\begin{description}
\item From \autoref{eq:operatorT_1} it follows
      \begin{equation}\label{eq:UniversalBP-Attenuation}
        h = \Bp{ \partial_t (\Id + \Top)^{-1}[(t,\vxi) \to e^{\kinf t}\qa(t,\vxi)]}.
      \end{equation}
\item In particular, in the case of a constantly attenuating medium the kernel of the integral
      operator $\AttOp{\cdot}$ simplifies to ($ k_*(\omega)=0$)
      \begin{equation*}
      \frac{1}{\sqrt{2\pi}} \IFf{e^{\i \kappa(\cdot)\tau}}(t) = \frac{1}{\sqrt{2\pi}} e^{-\kinf \tau} \IFf{e^{\i \cdot \tau}}(t)
      = e^{-\kinf \tau}\delta(t-\tau).
      \end{equation*}
      Thus from \autoref{eq:kernel} it follows that
      \begin{equation}
      \label{eq:operatorT_w}
      \boxed{
       \AttOp{q}(t,\vx) = \frac{1}{\sqrt{2\pi}}  e^{-\kinf t} q(t,\vx)\;.}
      \end{equation}
      Thus the operator $\AttOp{\cdot}$ is a multiplication operator and the reconstruction formula
      \autoref{eq:UniversalBP-Attenuation} rewrites to
      \begin{equation}\label{eq:UniversalBP-ConstAttenuation}
       h=\Bp{(t,\vxi) \to \partial_t  \left( e^{\kinf t}\qa(t,\vxi) \right)}.
      \end{equation}
      Using that $q = e^{\kinf t}\qa$, we get explicit formulas for the time derivatives of $q$:
      \begin{equation*} \begin{aligned}
      \partial_t q(t, \vxi)&= e^{\kinf t}\left(\kinf \qa(t, \vxi)+\partial_t \qa(t, \vxi)\right),\\
      \partial_{tt} q(t, \vxi) &= e^{\kinf t}\left(\kinf^2 \qa(t, \vxi)+ 2\kinf \partial_t \qa(t, \vxi)+\partial_{tt} \qa(t, \vxi)\right).
\end{aligned} \end{equation*}
Therefore, inserting the representation of the derivatives in \autoref{eq:UniversalBP-3D} we get

\begin{equation}\label{eq:UniversalBackProjection}
\boxed{
\begin{aligned}
\tilde{\qa} (t,\vxi) =& \partial_t \qa(t,\vxi) - t \partial_{tt} \qa(t,\vxi)\,,\\
h(\vx) =& \frac{2}{\Omega_0} \int\limits_{\vxi \in \Gamma} \frac{\tilde{\qa} (|\vxi-\vx|,\vxi)}{|\vxi-\vx|^2} \left(\mathbf{n}_{\vxi} \cdot \frac{\vxi-\vx}{|\vxi-\vx|}\right) \d s(\Gamma).
\end{aligned}}
\end{equation}
\end{description}

\section{Numerical experiments}
\label{sec:num}
In this section we describe an algorithm for photoacoustic inversion in a weakly attenuating
medium, in which case the attenuation coefficient is $\omega \to \kappa(\omega) = \omega +\i \kinf + k_*(\omega)$,
with $k_* \in L^2(\R;\C) \cap C^\infty(\R;\C)$.

The numerical inversions and examples will be performed in $\R^2$ for the two-dimensional attenuated wave equation.
This is consistent with a distributional solution of \autoref{eq:WEA3} in $\R^3$ where
$(x,y,z) \to h(x,y)$ is considered a distribution in $\R^3$, which is independent of the third variable.
In this case $\pa$ can be considered a two-dimensional distribution \autoref{eq:WEA3}, which solves the
two-dimensional attenuated wave equation:
\begin{equation}\label{eq:WEA2}
\begin{aligned}
\pseudo[\pa](t,x_1,x_2)-\Delta \pa(t,x_1,x_2)&=\delta'(t)h(x_1,x_2),\quad &&t\in\R,\;(x_1,x_2)\in\R^2, \\
\pa(t,x_1,x_2)&=0,\quad&&t<0,\;(x_1,x_2)\in\R^2.
\end{aligned}
\end{equation}
The two-dimensional universal back-projection formula from \cite{BurGruHalNusPal07}, which is used below,
is given by
\begin{equation}\label{eq:UniversalBP-2D}
\begin{aligned}
h(x)&=-\frac{4}{\Omega_0} \int\limits_{\vxi \in \Gamma} \int\limits_{t=|\vxi-\vx|}^{\infty}
      \left(\frac{(\partial_t (t^{-1}p))(t, \vxi)}{\sqrt{t^2-|\vxi-\vx|^2}}\d t\right) \mathbf{n}_{\vxi}\cdot (\vxi-\vx)\d s(\Gamma),
\end{aligned}
\end{equation}
where $\Omega_0$ is $2\pi$ for a line measurement geometry and $4\pi$ for a circular measurement geometry,
$\mathbf{n}_{\vxi}$ is the outer normal vector for the curve $\Gamma$.

We assume that the attenuated photoacoustic pressure $\pa$ is measured on a set of $N$ points on a measurement
curve $\Gamma$ at $N_T$ uniformly distributed time points
\begin{equation*}
t_i = i \Delta_T, \; i=1,\ldots,N_T \; \text{ where } \Delta_T = \frac{T}{N_T}.
\end{equation*}
In our experiments $\Gamma$ will either be a circle of radius $R$, where the $N$ measurement points are radially uniformly
distributed,
\begin{equation*}
\vxi_j = R(\cos(j\Delta_\xi), \sin(j\Delta_\xi)),\; j=0,1,\ldots,N-1\text{ where } \Delta_\xi = \frac{2\pi}{N}
\end{equation*}
or on a segment of length $2l$ of the $x$-axis, in which case
\begin{equation*}
\vxi_j= (2j\Delta_x-1,0), \text{ where } \Delta_x = l/N.
\end{equation*}
In this case we consider $h$ to be supported in the upper half-space.

The evaluation of the integral operator $\opint{.}$ is numerically realized as follows: For every measurement point
$\set{\vxi_j: j=0,1,\ldots,N-1}$
\begin{equation}\label{eq:dis_relation_pq}
\qa(t_i, \vxi_j) = \Delta_T \sum_{n=1}^{i} \pa(t_n,\vxi_j).
\end{equation}
The relation $\qa = \AttOp{q}$ from \autoref{eq:qa_q} is realized numerically as follows: Because we assume a weakly
attenuation medium $\AttOp{\cdot}$ (defined in \autoref{eq:mult}) is an integral operator with kernel $k$ defined
in \autoref{eq:tau}. We use the Taylor-series expansion of the exponential function $\tau \to e^{\i k_*(\omega) \tau}$
and get
\begin{equation}\label{eq:Fourier_rj}
\IFf{e^{\i k_*(\omega)\tau}-1}(t)=\sum_{k=1}^{\infty}\frac{\tau^k}{k!}\IFf{(\i k_*(\omega))^k}(t).
\end{equation}
Inserting \autoref{eq:Fourier_rj} into \autoref{eq:top}  and taking into account \autoref{eq:operatorT_1} and \autoref{eq:qa_q} we get for all
$i=1,\ldots,N_T$ that
\begin{equation}\label{qa-texpress2}
\begin{aligned}
 ~ & \qa(t_i,\vxi_j) \\
 = & e^{- \kinf t_i} q(t_i,\vxi_j) +\frac{1}{\sqrt{2\pi}}\intb{\tau}
e^{-\kinf \tau}\sum_{k=1}^{\infty}\frac{\tau^{k}}{k!}r_k(t_i-\tau)q(\tau, \vxi_j)\d \tau,
\end{aligned}
\end{equation}
where
\begin{equation*}
s \to r_k(s) := \IFf{\i^k k_*^k(\omega)}(s).
\end{equation*}
The integral on the right hand side of \autoref{qa-texpress2} is approximated for numerical purposes as follows:
\begin{equation}\label{qa-texpress3}
\frac{\Delta_T}{\sqrt{2\pi}} \sum_{m=1}^{N_T} e^{-\kinf t_m} \sum_{k=1}^{\infty}
\frac{t_m^{k}}{k!}r_k(t_i-t_m)q(t_m, \vxi_j).
\end{equation}
This expression is represented as a matrix-vector multiplication with
vector $\vec{q}_j = (q(t_m,\vxi_j))_{m=1,\ldots,N_T}$ and matrix with entries
\begin{equation}\label{eq:MatrixElement}
b_{im} = \frac{\Delta_T}{\sqrt{2\pi}} e^{-\kinf t_m} \sum_{k=1}^{\infty} \frac{t_m^{k}}{k!}r_k(t_i-t_m) \text{ with }
 1 \leq i \leq N_T,\text{ and } 1 \leq m \leq N_T.
\end{equation}
Then it follows from \autoref{qa-texpress2} that
\begin{equation}
\label{eq:matrix}
\vec{q}_j^a = (e^{-\kinf t_i} I + B) \vec{q}_j\,,
\end{equation}
which is the discretized version of \autoref{eq:operatorT_1}.
To get numerical values for the entries of $B$, the terms $r_k(t_i-t_m)$ have to be numerically calculated. For $k=1$,
\begin{equation}\label{eq:dis_r1}
r_1(t)=\frac{1}{\sqrt{2\pi}} \intb{\omega} \i k_*(\omega)e^{-\i \omega t}\d \omega,
\end{equation}
which can be evaluated by numerical integration for all $t_i$. When $k>1$, $r_{k}$ is a convolution of $r_{k-1}$ and $r_1$
and thus
\begin{equation}\label{eq:dis_rj}
r_{k}(t) = \frac{1}{\sqrt{2\pi}} (r_{1}*r_{k-1})(t) = \frac{1}{\sqrt{2\pi}} \int\limits_{0}^{t}r_{1}(\tau)r_{k-1}(t-\tau)\d \tau.
\end{equation}
Numerically, we approximate the convolution by
\begin{equation*}
r_{k}(t_i)\approx \frac{\Delta_T}{\sqrt{2\pi}} \sum_{m=1}^{i}r_1(t_m)r_{k-1}(t_i-t_m).
\end{equation*}

We summarize the inversion process in a pseudo-code, where we truncate the Taylor-series \autoref{eq:Fourier_rj} at
the tenth coefficient (the number ten has been found from numerical simulations):
\begin{center}
\begin{algorithm}[H]
\SetAlgoLined
\KwData{The measurements are denoted by $P^a_{i,j}=\pa(t_i,\vxi_j)$ for all $i=1,\ldots,N_T$ and $j=0,\ldots,N-1$}
\KwResult{Numerical calculation of the absorption density $h_l=h(x_l)$}

\For{$1\leq i\leq N_T$}{
$r_{i,1} \leftarrow \frac{1}{\sqrt{2\pi}} \intb{\omega} e^{-\i \omega t_i} (\i k_*(\omega)) \d\omega$\;
}

\For{$1 \leq k\leq 10$}{
\For{$1 \leq m\leq n\leq N_T$}{
$(F_k)_{i,m}\leftarrow \frac{1}{\sqrt{2\pi}} r_{i-m, k}t_m^k e^{-\kinf t_m}$\;
}
\For{$1\leq i\leq N_T$}{
$r_{i, k+1}=\frac{\Delta_T}{\sqrt{2\pi}}\cdot(\sum_{m=1}^{i}r_{m, k}r_{i-m,k})$\;
}
}
$B\leftarrow \operatorname{diag}(e^{-\kinf t_1},e^{-\kinf t_2},\dots)+\Delta_T \sum_{k=1}^{10}\frac{F_k}{k!}$\;

\For{$1\leq i\leq N_T$}{
$Q^a_{i,j}\leftarrow Q^a_{i-1,j}+\Delta_T P^a_{i,j}$\;
}
$Q\leftarrow Q^aB^{-1}$\;
\For{$1\leq i\leq N_T$}{
$P_{i,j}\leftarrow \frac{Q_{i,j}-Q_{i-1,j}}{\Delta_T}$\;
}
Calculate $h_l$ by applying the back-projection operator $\Bp{\cdot}$ on $P_{i,j}$\;
\caption{Pseudocode for reconstructing the absorption density $h$.}
\label{pseudo-code}
\end{algorithm}
\end{center}

\subsection*{Numerical experiments}
We assume that $h$ is a function with compact support in $\R^2$. We calculated $p$, the solution of \autoref{eqWaveEquation}
using the k-wave toolbox \cite{TreCox10}. By integrating $p$ at the points $\vxi_j$, $j=0,\ldots,N-1$ over time with \autoref{eq:dis_relation_pq}
we get $q(t_i,\vxi_j)$ for $i=1,\ldots,N_T$ and $j=0,\ldots,N-1$. Then we find $\qa(t_i,\vxi_j)$
by matrix-vector multiplication \autoref{eq:matrix}.

In order to avoid \emph{inverse crimes} we used different discretization points in space and time for
the simulation of the forward data and the inversion. The forward problem is simulated with $N_T=500$ and
$N=896$, while the inverse problem is solved with $N_T=443$ and $N=849$.
The absorption density function $h:\R^2 \to \R$ with support in $(-0.8,0.8)^2$ is
the Shepp-Logan phantom \cite{SheLog74}.
In all numerical experiments the material parameter $\kinf = 0.45$.

\subsection*{Circular measurement geometry}
In these examples the measurement geometry is a circle with radius $R=1.7$ on which there are recorded
data on $N=849$ uniformly distributed measurement points. Moreover, the time length is $6$ and thus $\Delta_T = 6/N_T=6/443$.

We consider a constantly attenuating medium, with attenuation coefficient $\kappa(\omega) =\omega+ \i \kinf$.
\autoref{fig:const_circle} shows the ground truth (top left) and the simulated pressure data $\pa$ on $\Gamma$ over time.
Two reconstructions are presented: The first one is obtained by applying the universal back-projection formula
\autoref{eq:UniversalBP-2D}
(middle left), while the middle right image shows the reconstruction obtained with \autoref{pseudo-code}. The quantitative
values of ground truth and the two reconstructions are plotted on the bottom.
\begin{figure}[h]
  \centering
  \includegraphics[width=0.4\textwidth]{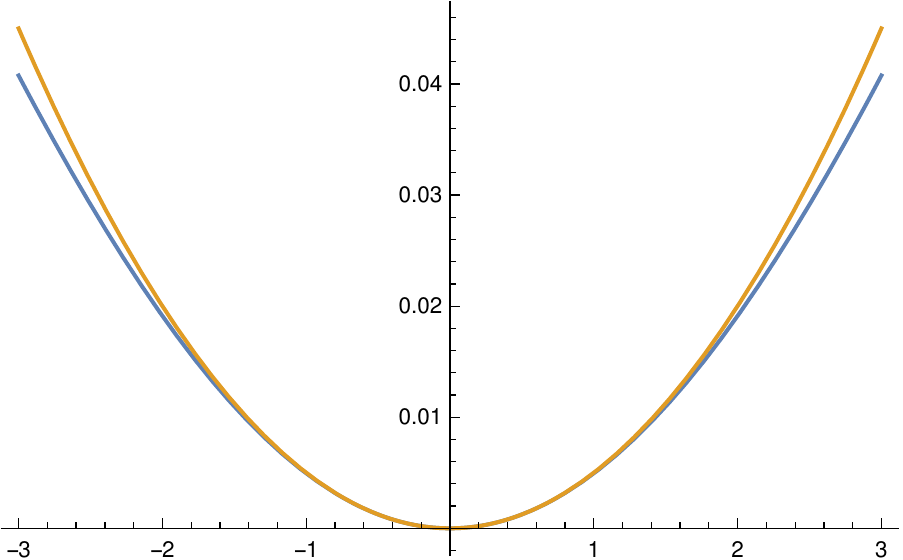}
  \caption{\label{fig:NSWImkappa} Blue curve corresponds to $\Im \kappa(\omega)$ of NSW model. Red curve corresponds to power law $0.005\omega^{2}$.}
\end{figure}
\begin{figure}[h]
  \centering
  \includegraphics[width=0.4\textwidth]{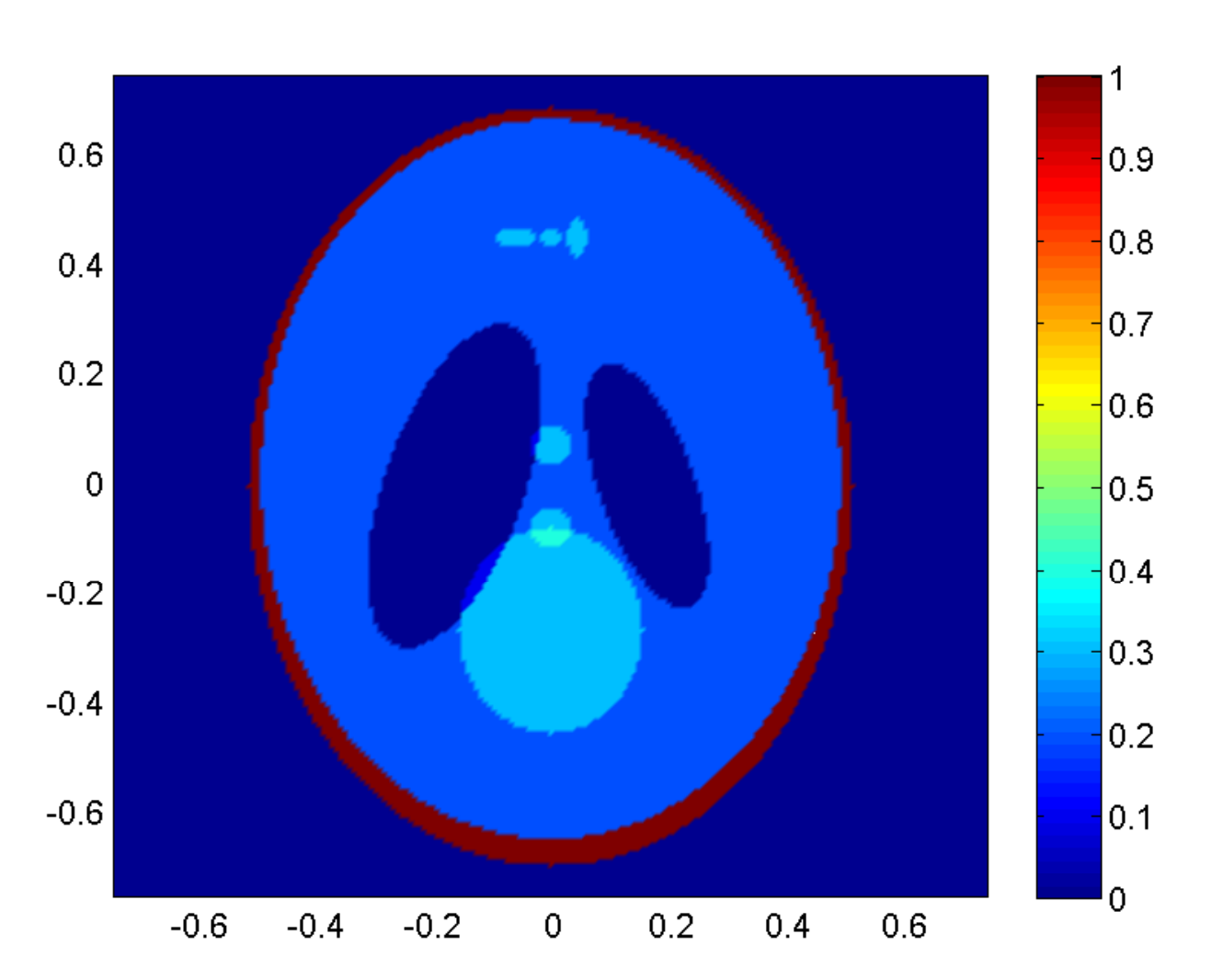}
  \includegraphics[width=0.4\textwidth]{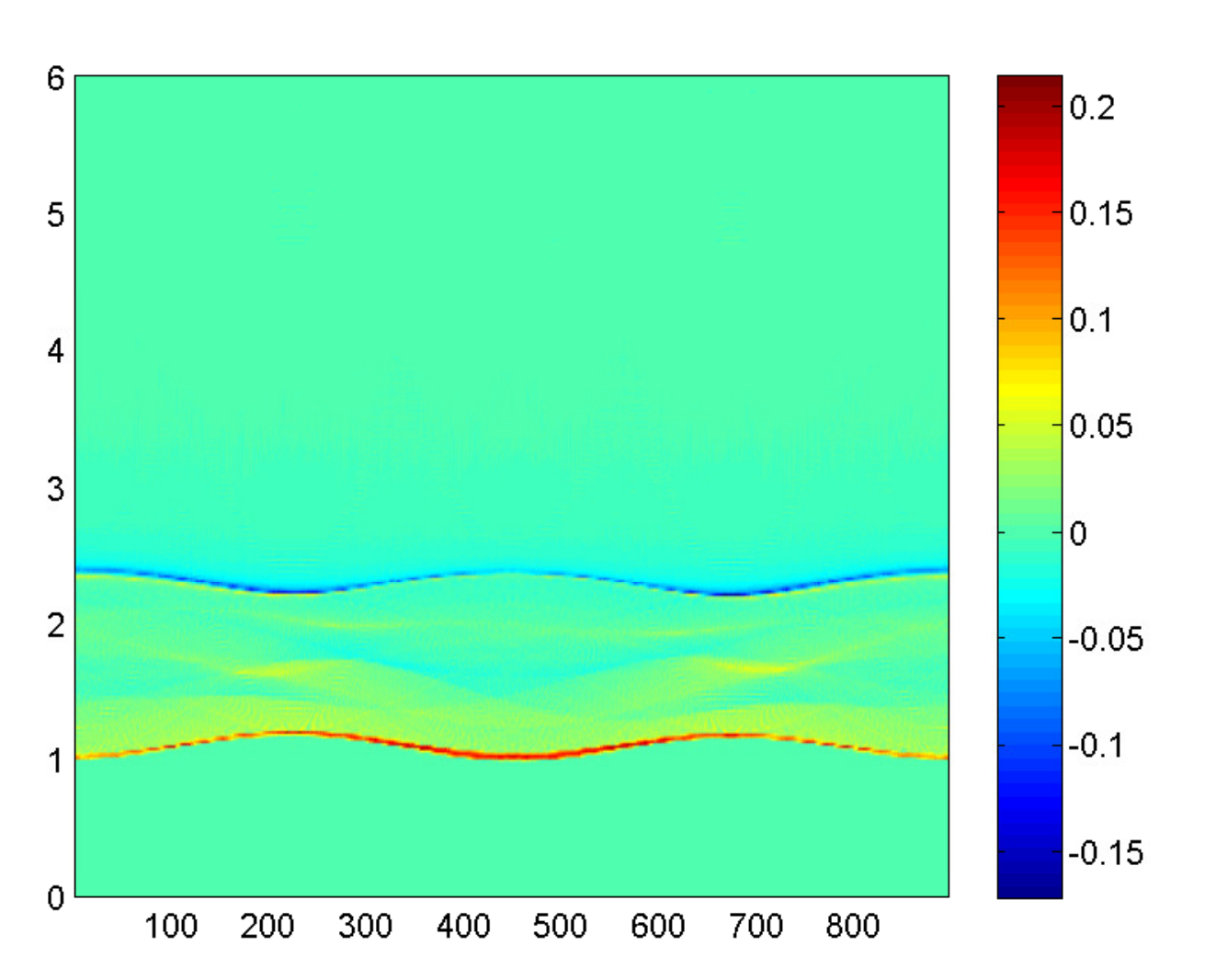}\\
  \includegraphics[width=0.4\textwidth]{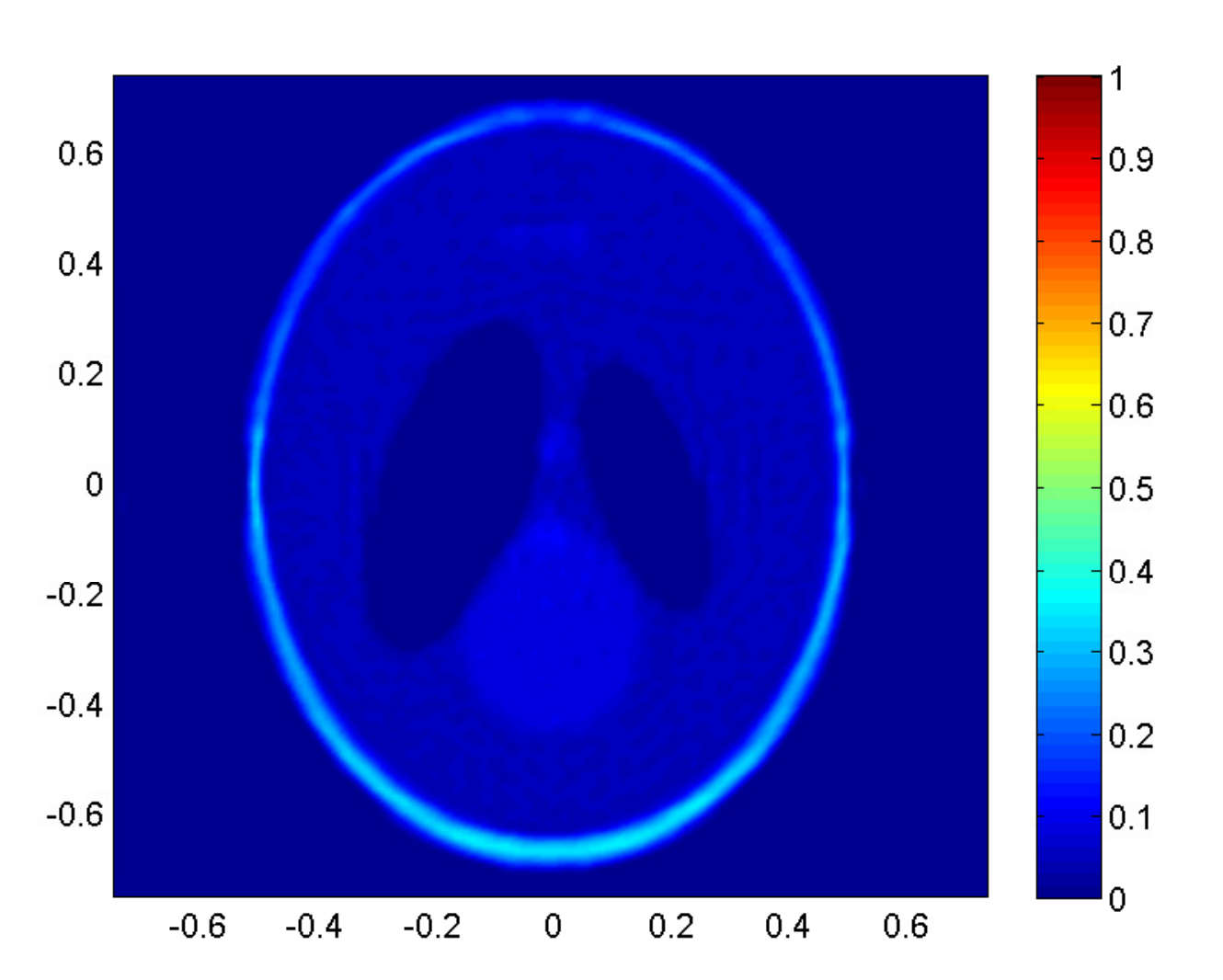}
  \includegraphics[width=0.4\textwidth]{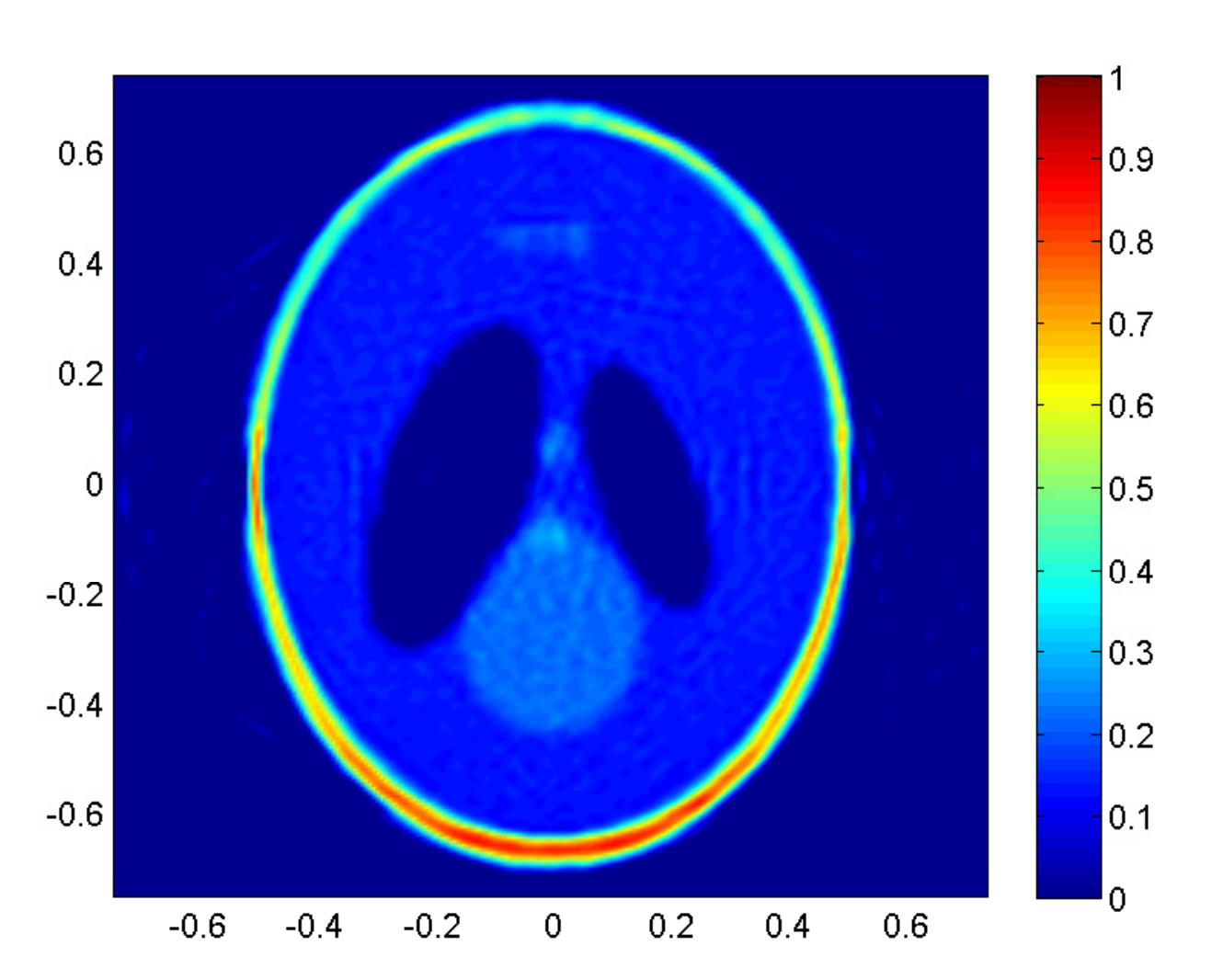}\\
  \includegraphics[width=0.4\textwidth]{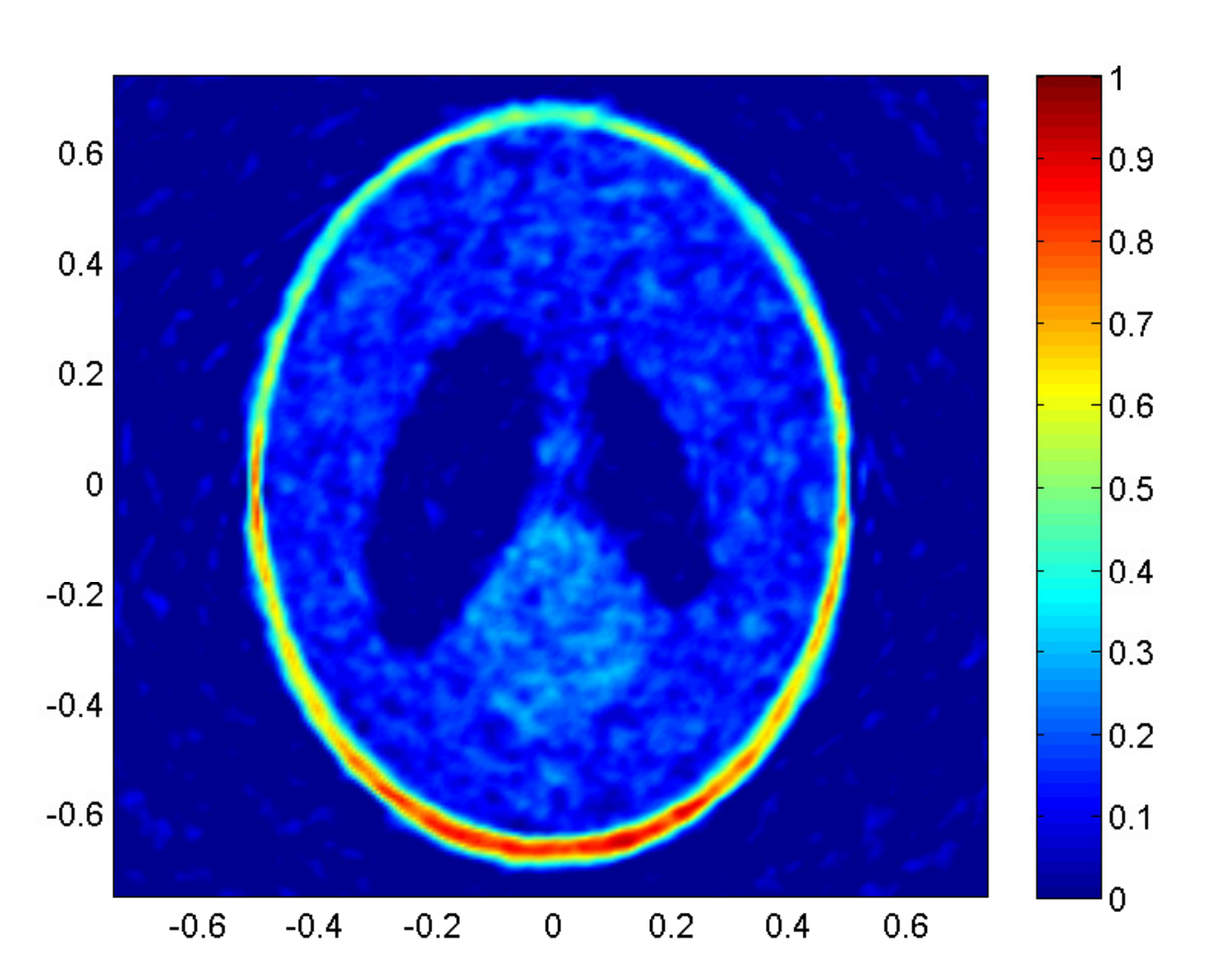}\\
  \includegraphics[width=0.7\textwidth]{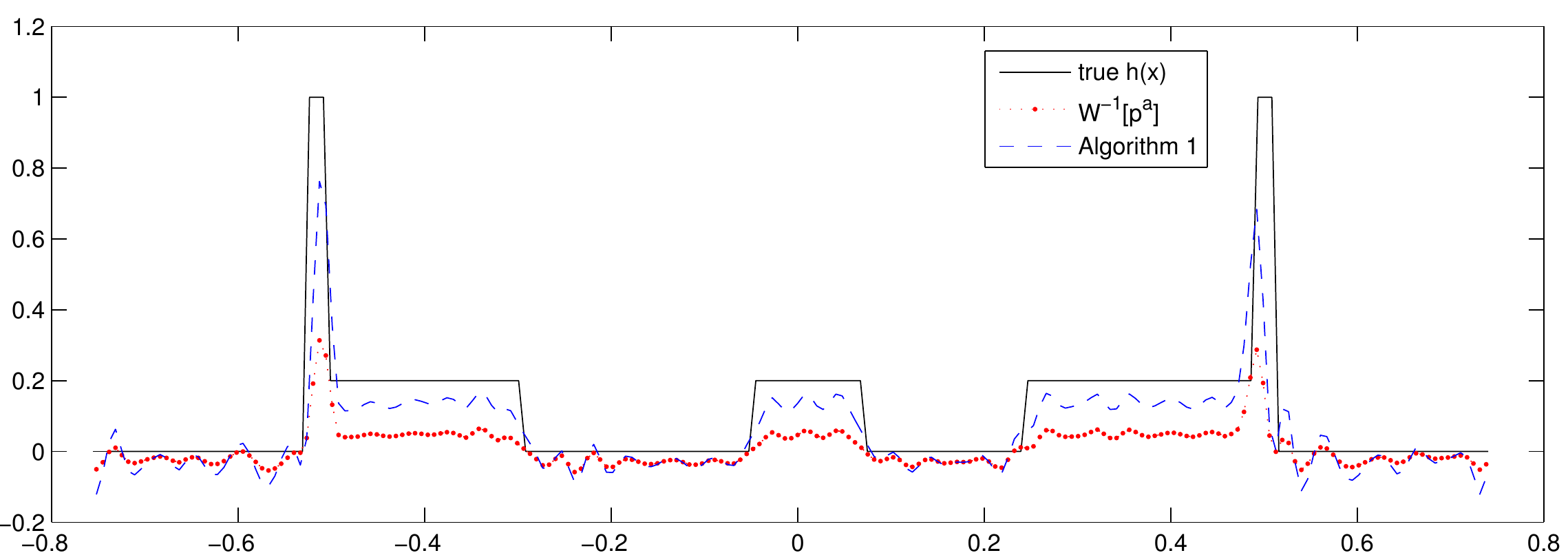}
  \caption{\label{fig:const_circle}Measurements along a circle and constantly attenuating model. Top left: Ground truth.
  Top Right: the simulated pressure data $\pa$. Middle: Reconstruction by universal back-projection
  (not taking into account attenuation), and by \autoref{pseudo-code} with noise-free and $20\%$ noise.
  Bottom: Cross section through ground truth and reconstructions.}
\end{figure}

Next we consider the  Nachman, Smith and Waag (NSW) \cite{NacSmiWaa90} attenuation model:
\begin{equation}\label{eq:NSW}
\kappa(\omega)=\omega \sqrt{\frac{1-\i \omega \tilde{\tau}}{1-\i \omega \tau}} =
\omega+\frac{\tau-\tilde{\tau}}{2\tau\tilde{\tau}}\i+k_*(\omega)
\end{equation}
where $k_*(\omega)=\mathcal{O}(|\omega|^{-1})$.
Therefore, $\kappa$ is a weak attenuation coefficient with $\kinf=\frac{\tau-\tilde{\tau}}{2\tau\tilde{\tau}}$.
In \autoref{fig:circular2} we present ground truth, simulated measurements $\pa$, and compare three imaging
techniques:
\begin{itemize}
\item Applying the universal back-projection formula $\Bp{\pa}$ \autoref{eq:UniversalBP-2D} (thus neglecting the attenuation).
\item The compensated back-projection formula
      \begin{equation}\label{eq:RescaledBp}
       \Bp{(t, \vxi)\mapsto \partial_t \left( e^{\kinf t}\qa(t, \vxi)\right)},
      \end{equation}
      which neglects $k_*(\omega)$ but takes into account $\kinf$.
\item Reconstruction using \autoref{eq:explicit} with the numerical code described in \autoref{pseudo-code}.
\end{itemize}
The parameters of attenuation coefficient in the NSW model are $\tilde{\tau}=0.1$ and $\tau=0.11$. For small frequencies the NSW
coefficients behaves like a power law of order $2$. However, asymptotically, for large frequencies,
it behaves like $\omega + \frac{\tau -\tilde{\tau}}{2 \tau \tilde{\tau}} \i$. The NSW-coefficient has been plotted
in \autoref{fig:NSWImkappa}.
In order to demonstrate the stability of the algorithm, we also performed reconstructions from noisy data,
where a uniformly distributed noise is added with a variance of $20\%$ of the maximal intensity.
The reconstruction results are depicted in the last image of \autoref{fig:const_circle} and \autoref{fig:circular2},
respectively.

\begin{figure}[h]
  \centering
  \includegraphics[width=0.4\textwidth]{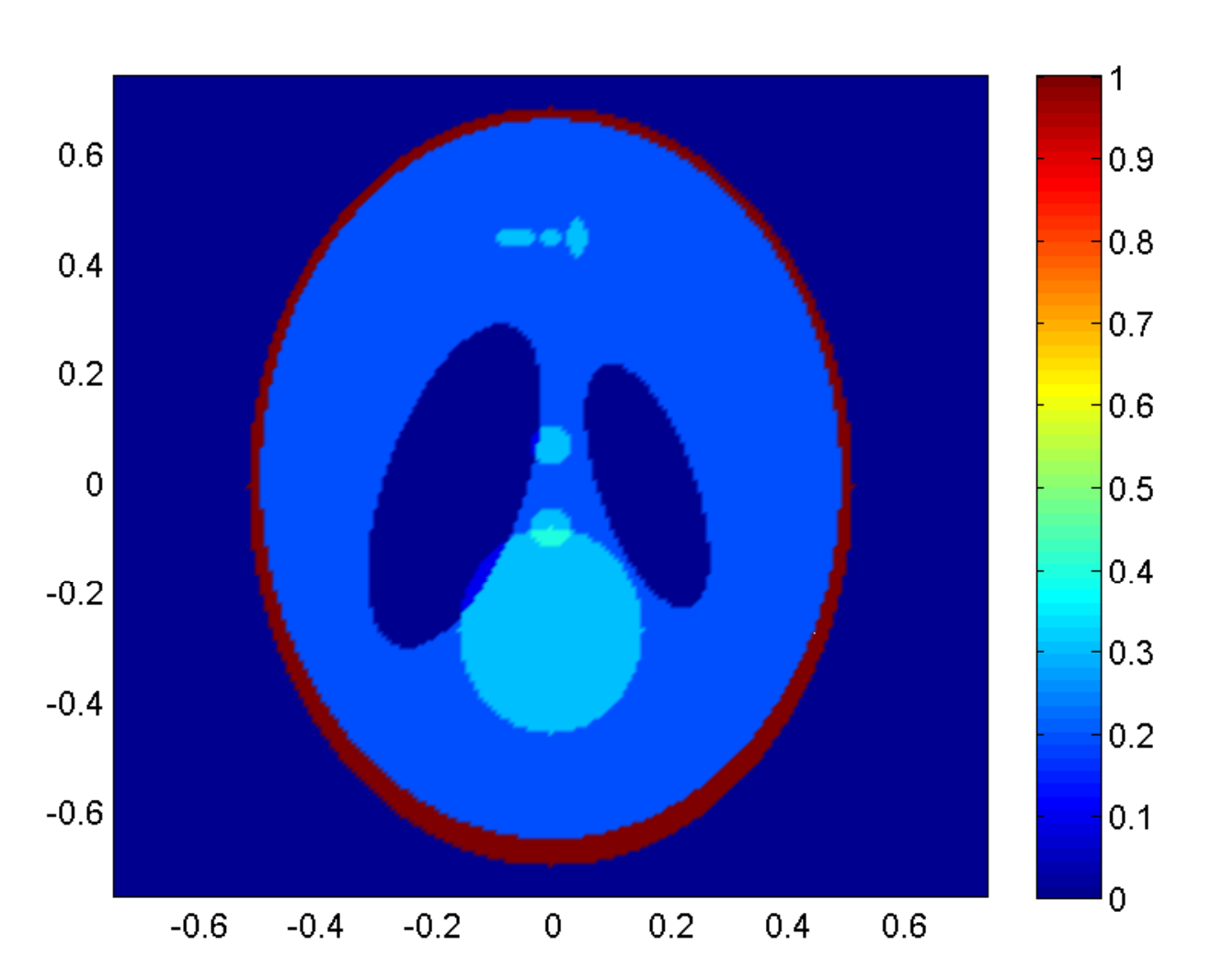}
  \includegraphics[width=0.4\textwidth]{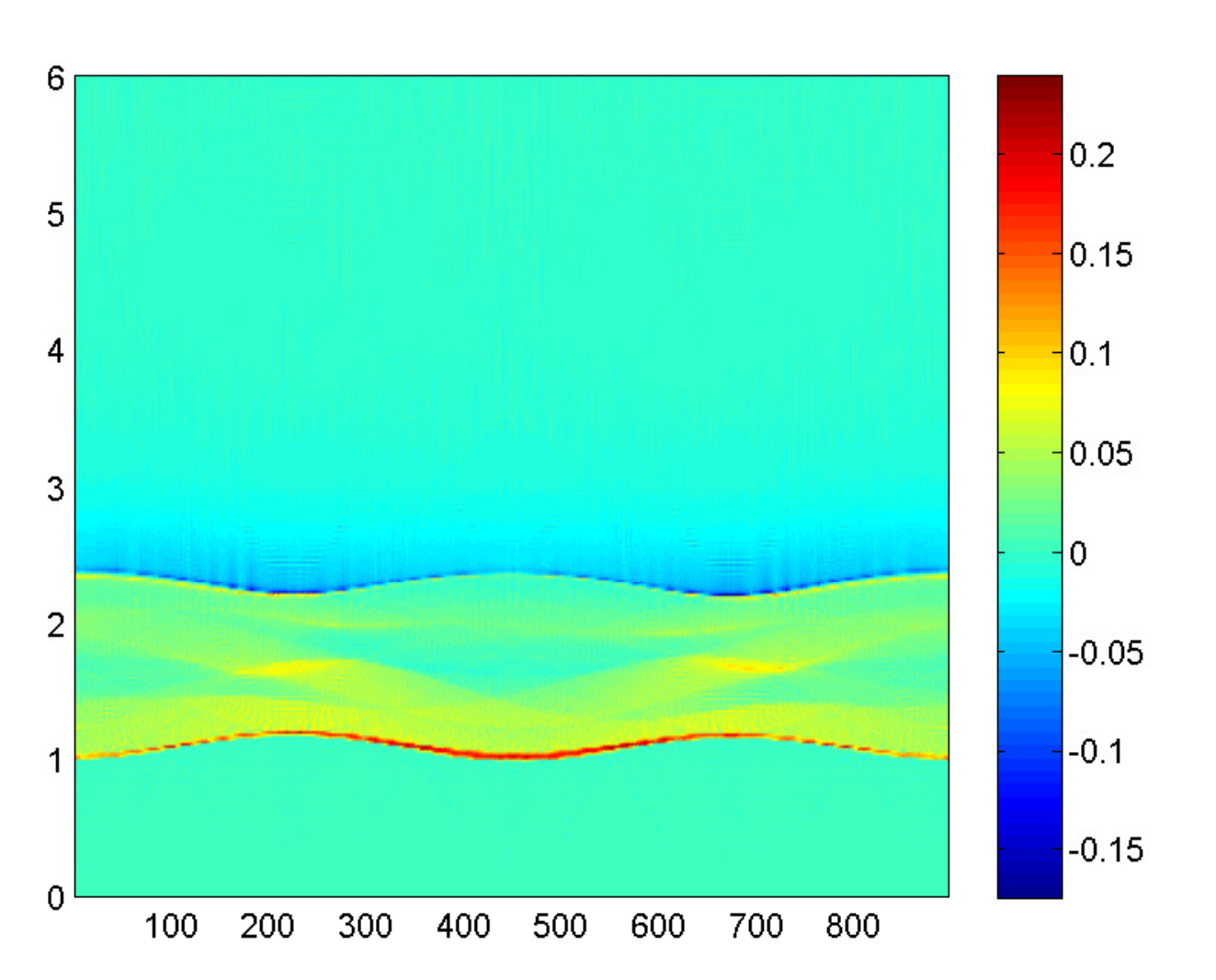}\\
  \includegraphics[width=0.4\textwidth]{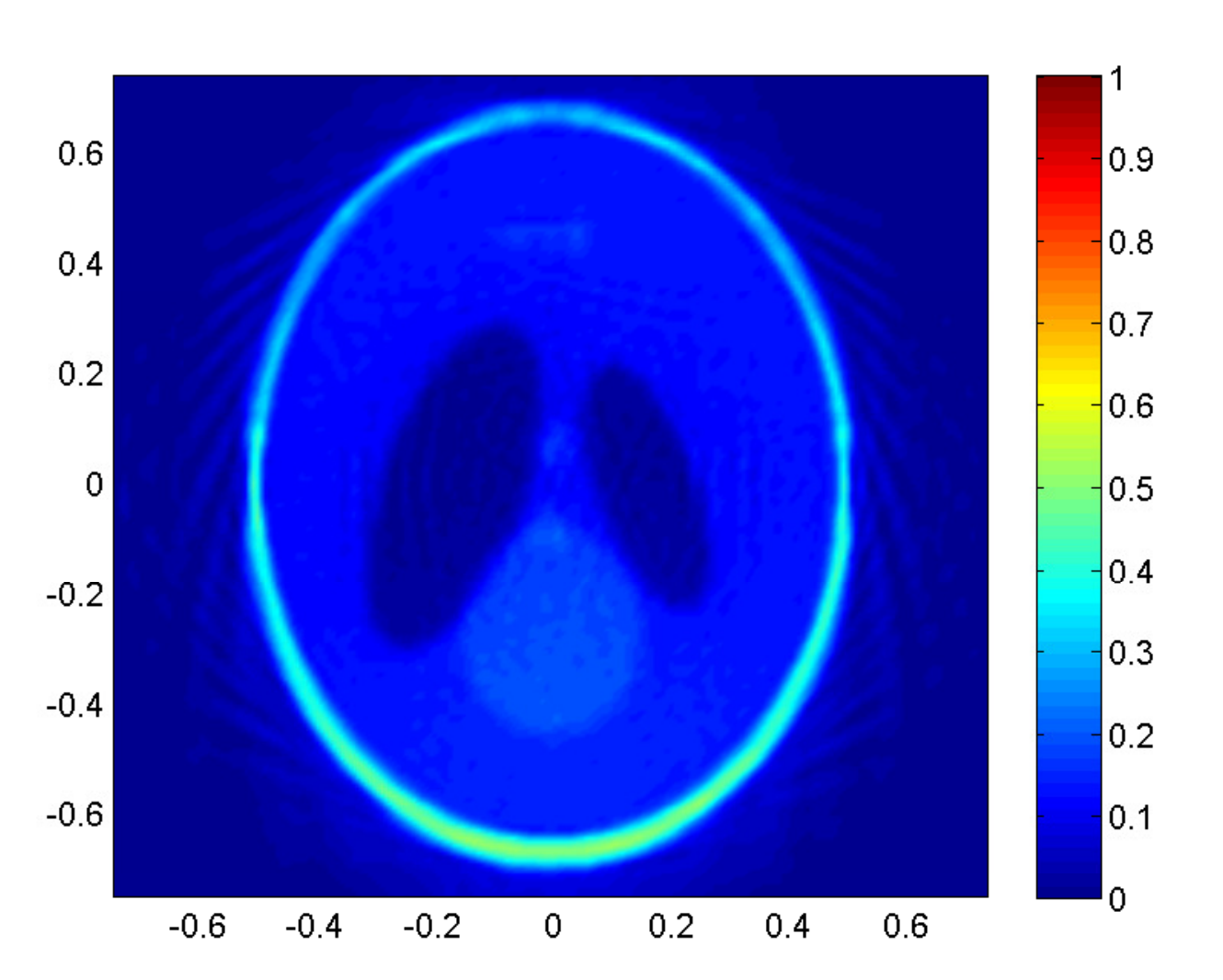}
  \includegraphics[width=0.4\textwidth]{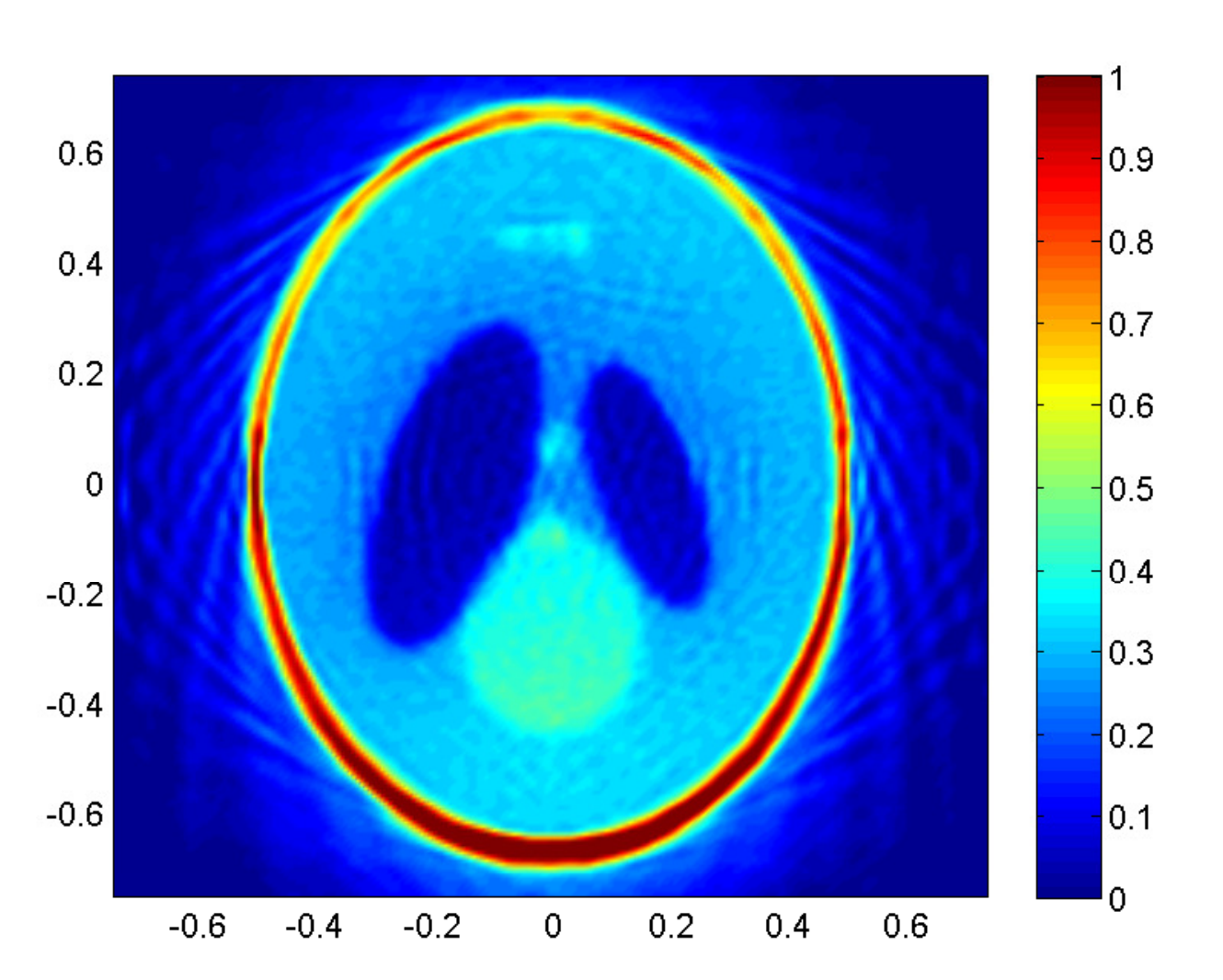}\\
  \includegraphics[width=0.4\textwidth]{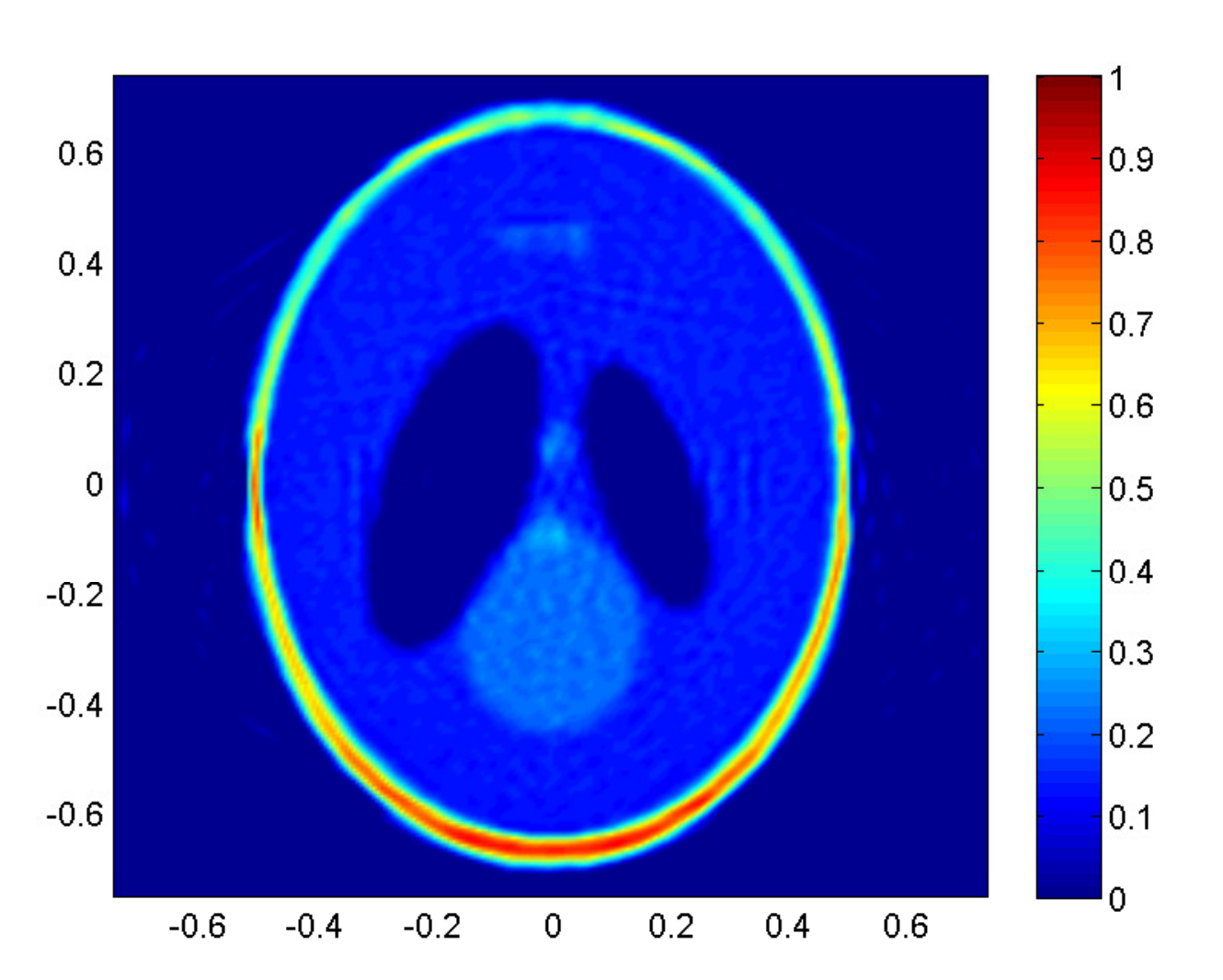}
  \includegraphics[width=0.4\textwidth]{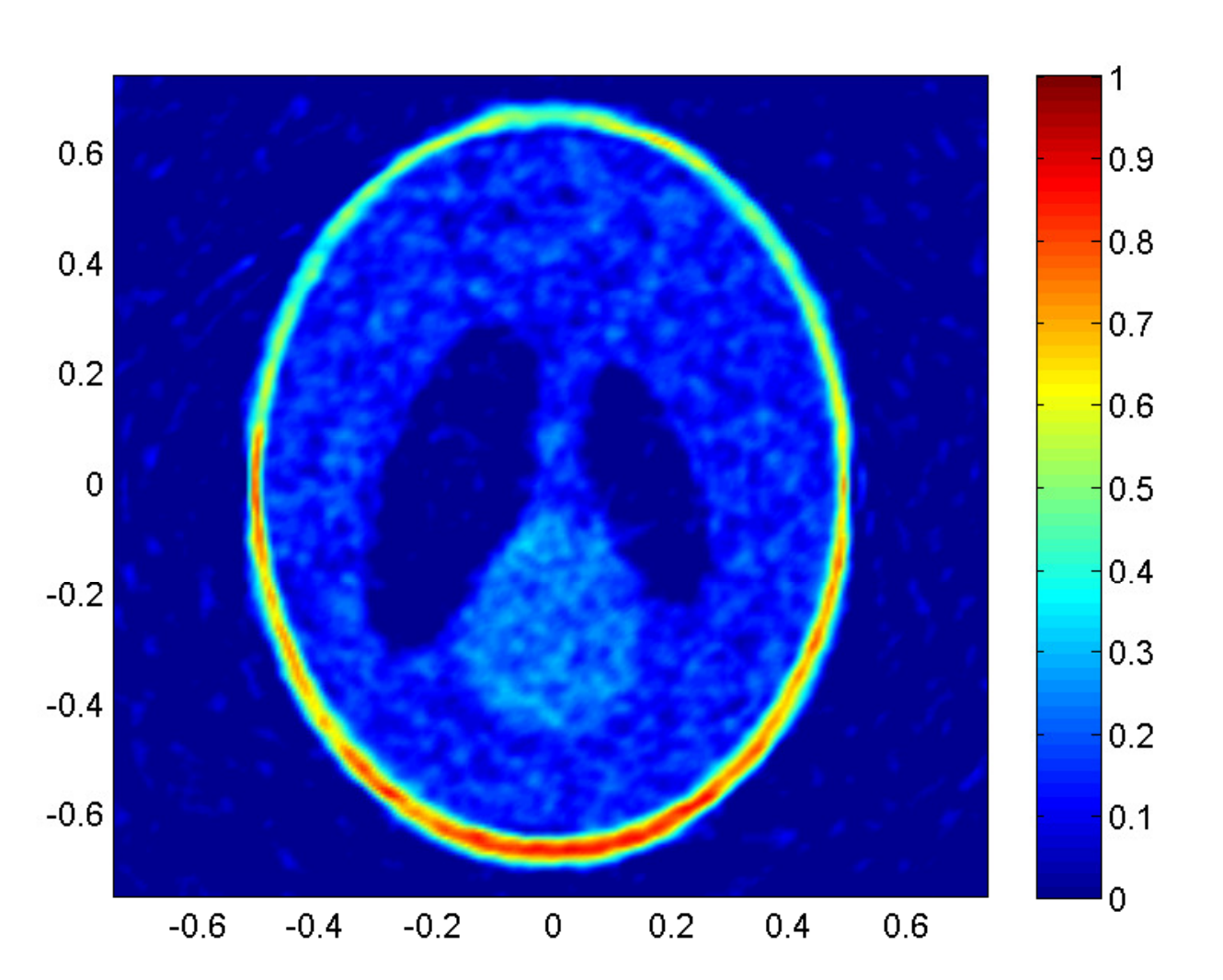}\\
  \includegraphics[width=0.7\textwidth]{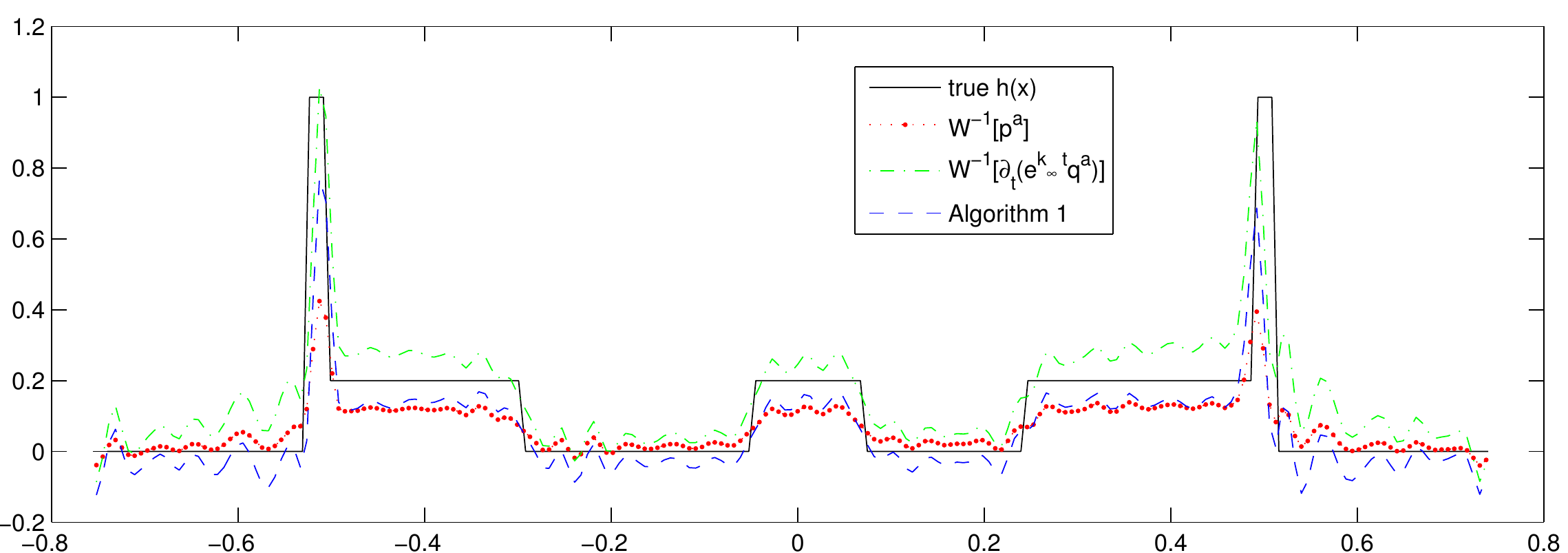}
  \caption{\label{fig:circular2}Measurements along a circle and NSW model. Top left: Ground truth.
  Top Right: the simulated pressure data $\pa$. Middle: Reconstruction by universal back-projection
  (not taking into account attenuation), by compensated attenuation \autoref{eq:RescaledBp} and by \autoref{pseudo-code} with noise-free and $20\%$ noise.
  Bottom: Cross section through ground truth and reconstructions.}
\end{figure}

\clearpage

\subsection*{Measurements on a line}
The measurement points are $N=849$ uniformly distributed on a line segment with length $l=10.2$.
The distance of the line to the center of the phantom is $1.7$. The time length is $8$ and thus $\Delta_T = 8/443$.

We consider a constantly attenuating medium, with attenuation coefficient $\kappa(\omega) =\omega+ \i \kinf$.
\autoref{fig:line1} shows the ground truth (top left) and the simulated pressure data $\pa$ on $\Gamma$ over time.
Two reconstructions are presented: The first one is obtained by applying the universal back-projection formula
\autoref{eq:UniversalBP-2D}
(middle left) and the middle right image shows the reconstruction obtained with \autoref{pseudo-code}. The quantitative
values of ground truth and the two reconstructions are plotted on the bottom.

\begin{figure}[h]
  \centering
  \includegraphics[width=0.4\textwidth]{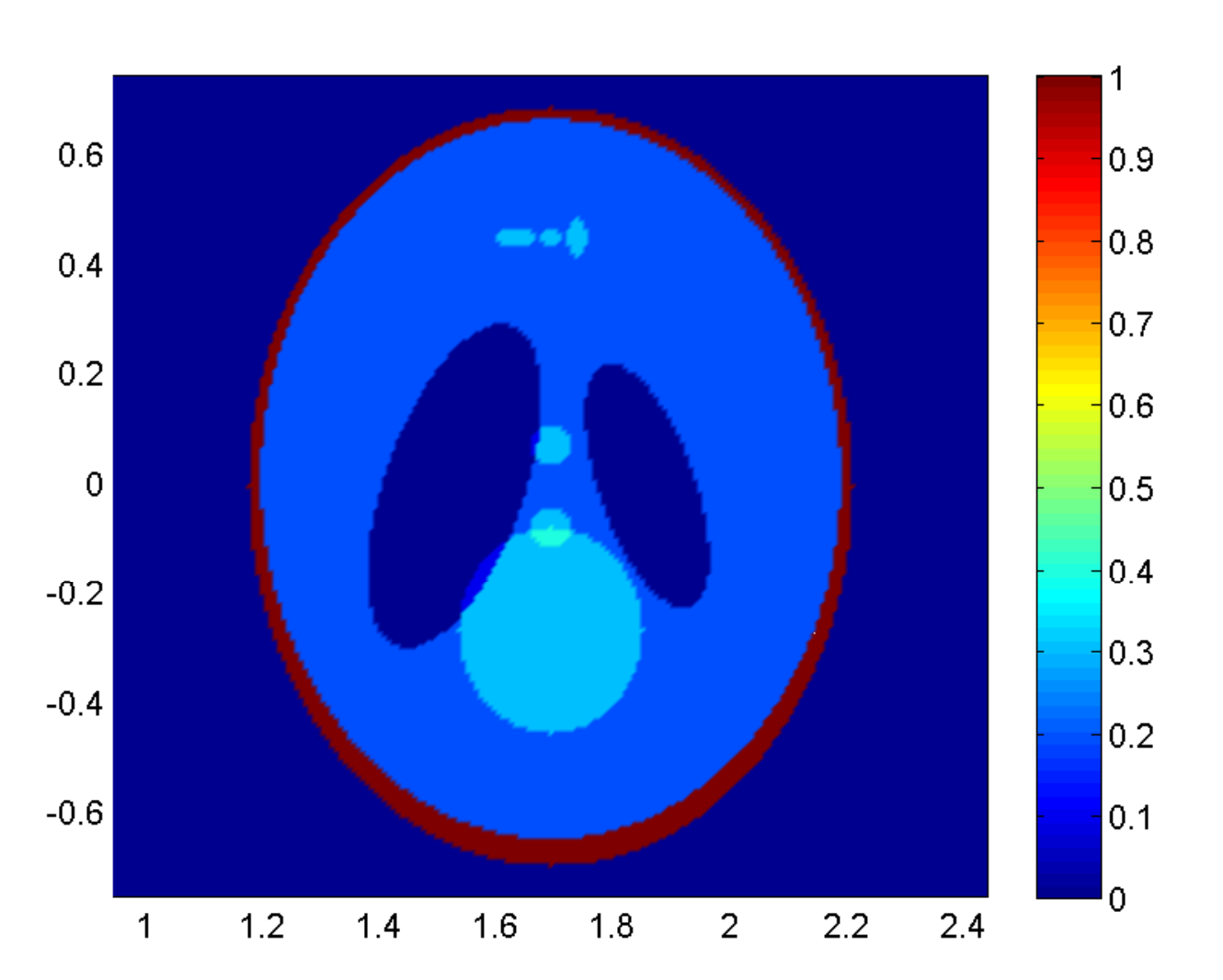}
  \includegraphics[width=0.4\textwidth]{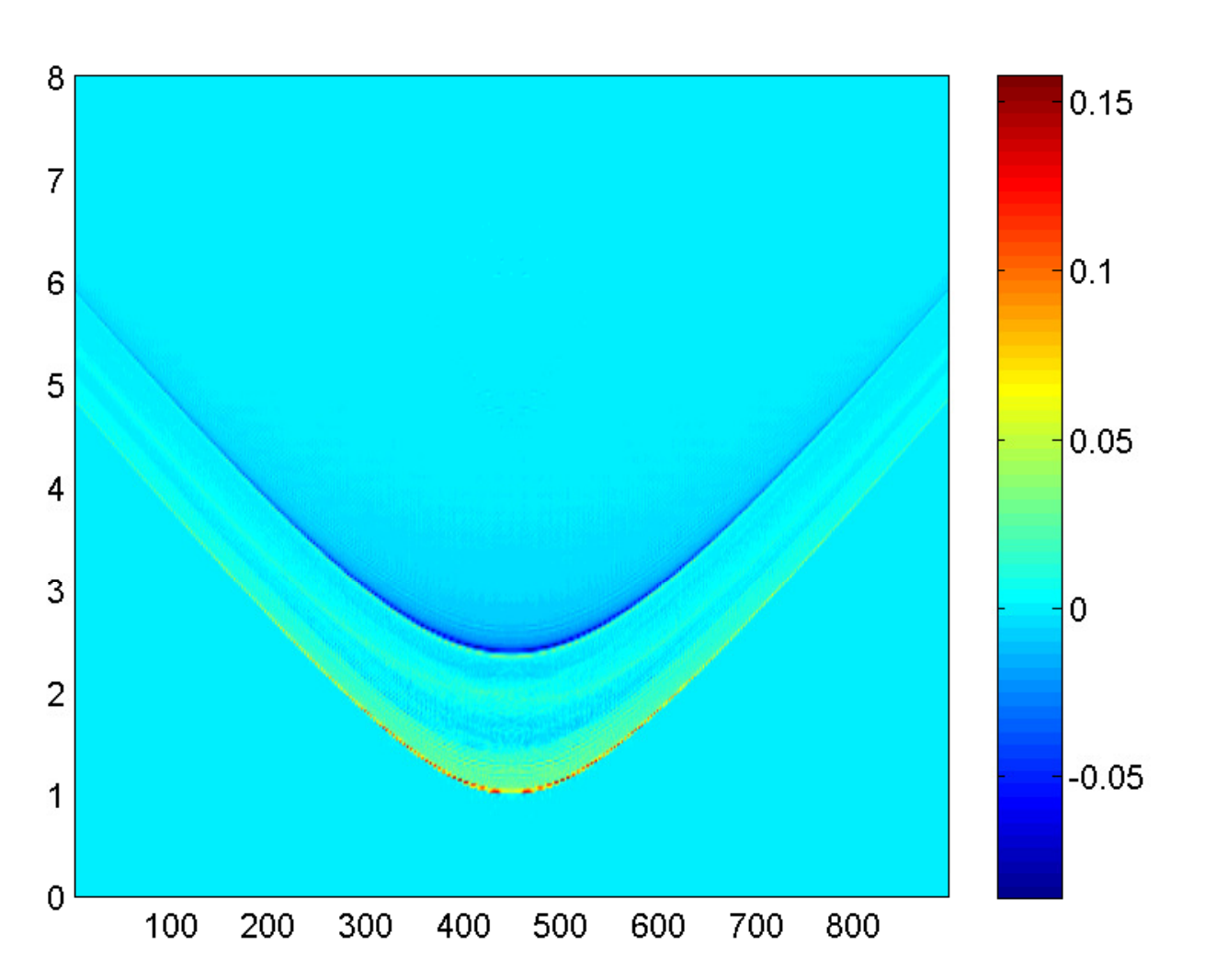}\\
  \includegraphics[width=0.4\textwidth]{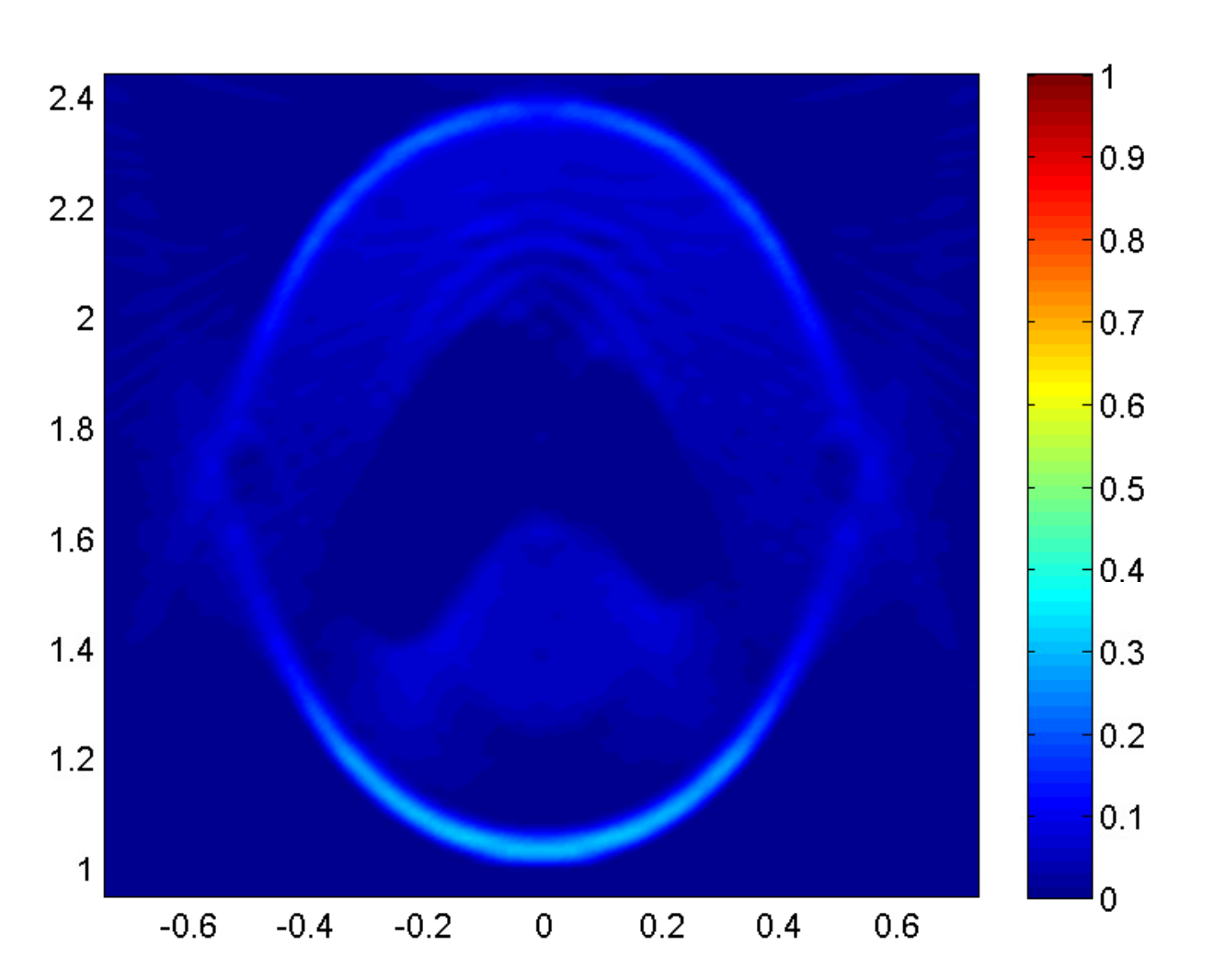}
  \includegraphics[width=0.4\textwidth]{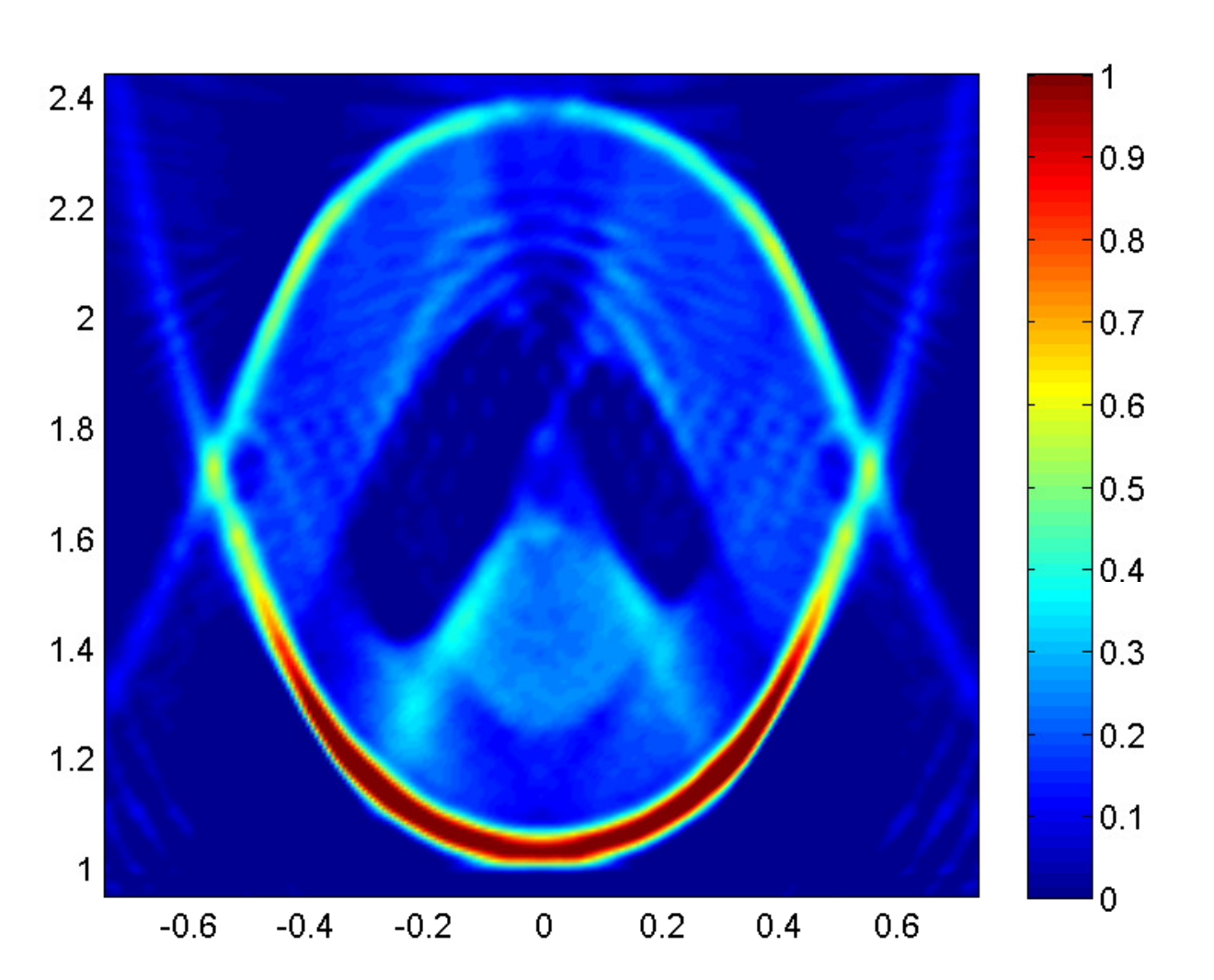}\\
  \includegraphics[width=0.4\textwidth]{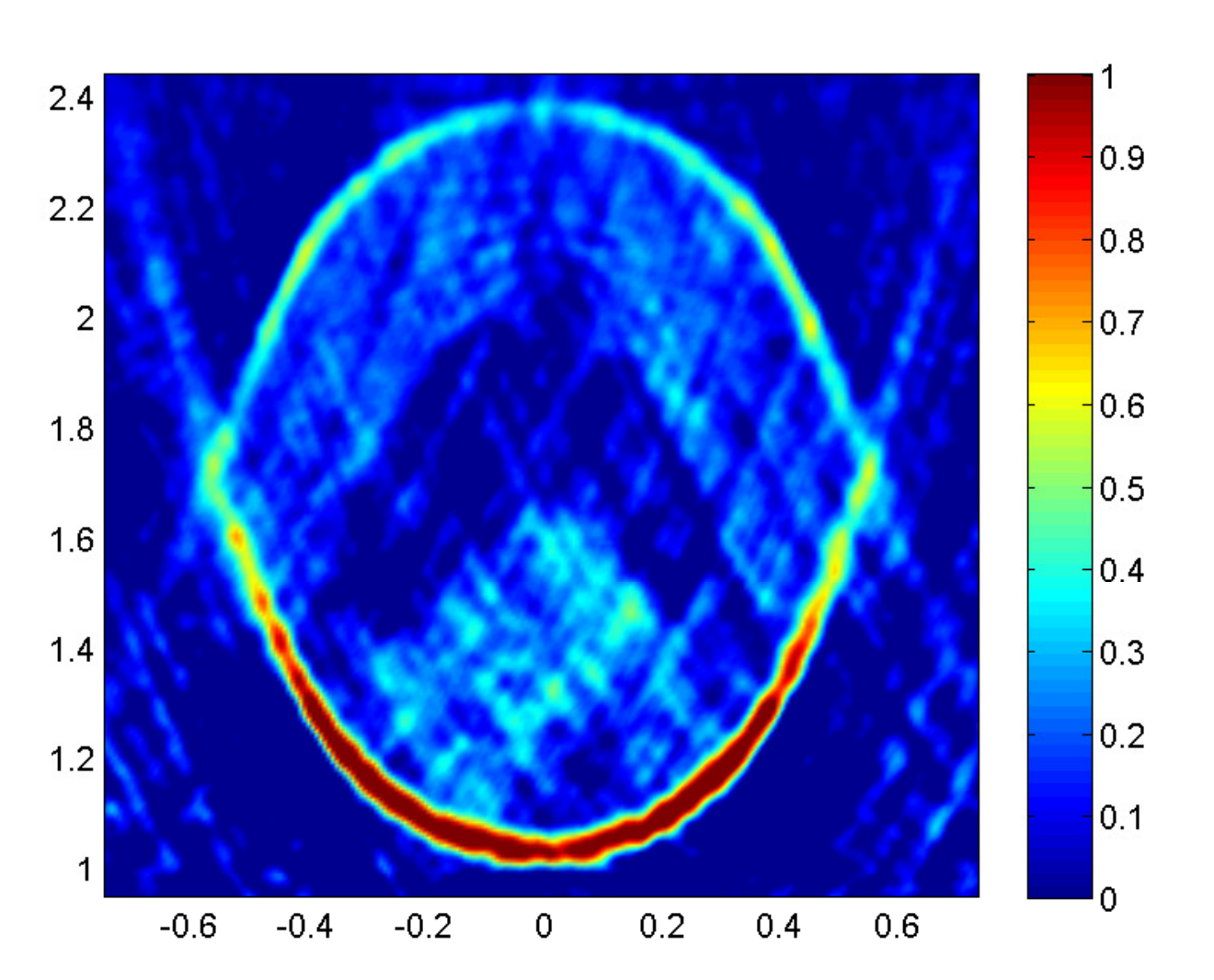}\\
  \includegraphics[width=0.7\textwidth]{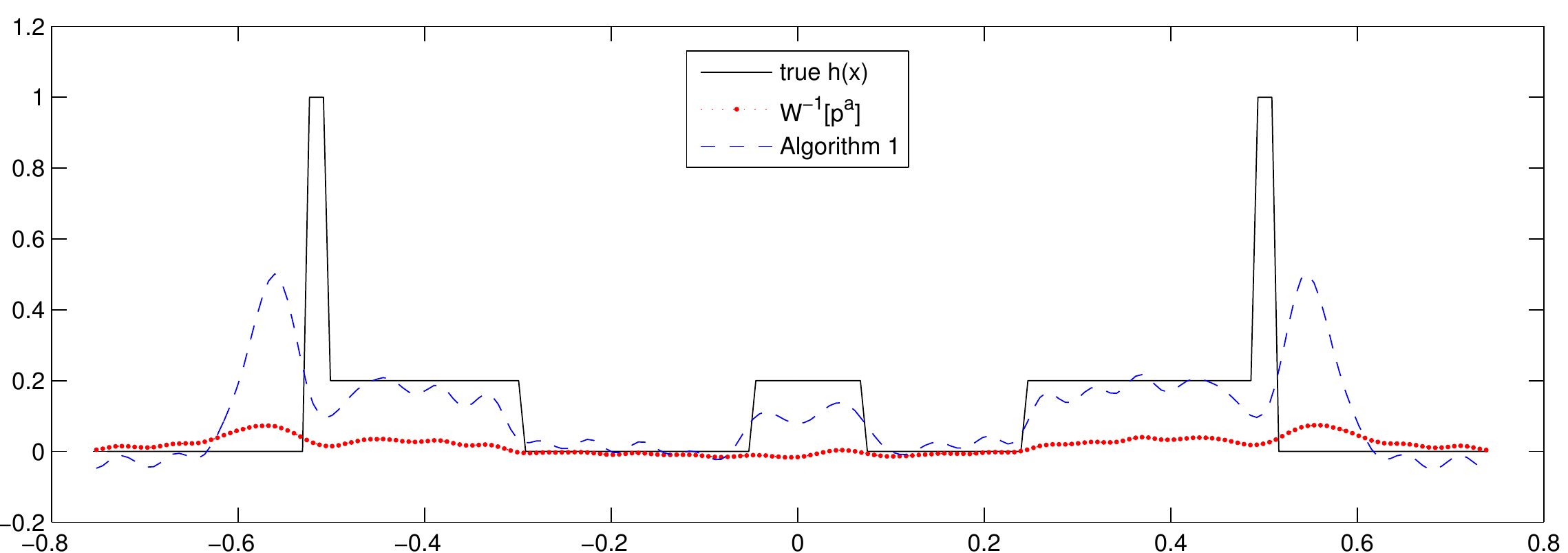}
  \caption{\label{fig:line1}Measurements along a line and constantly attenuating model. Top left: Ground truth.
  Top Right: the simulated pressure data $\pa$. Middle: Reconstruction by universal back-projection
  (not taking into account attenuation), and by \autoref{pseudo-code} with noise-free and $20\%$ noise.
  Bottom: Cross section through ground truth and reconstructions.}
\end{figure}

In \autoref{fig:line2} we present ground truth, simulated measurements $\pa$ using NSW model, and compare three imaging
techniques, the universal back-projection formula neglecting attenuation, the compensated back-projection formula
\autoref{eq:RescaledBp}, which neglects $k_*(\omega)$ but takes into account $\kinf$, and reconstruction with
\autoref{pseudo-code}. The parameters of the NSW attenuation coefficient are again $\tilde{\tau}=0.1$ and $\tau=0.11$.
The reconstruction results from noisy data are depicted in the last image of \autoref{fig:line1} and \autoref{fig:line2},
where uniformly distributed noise is added with a variance of $20\%$ of the maximal intensity value.
Numerical results show that the algorithm is quite stable even with $20\%$ noise.

\begin{figure}
  \centering
  \includegraphics[width=0.4\textwidth]{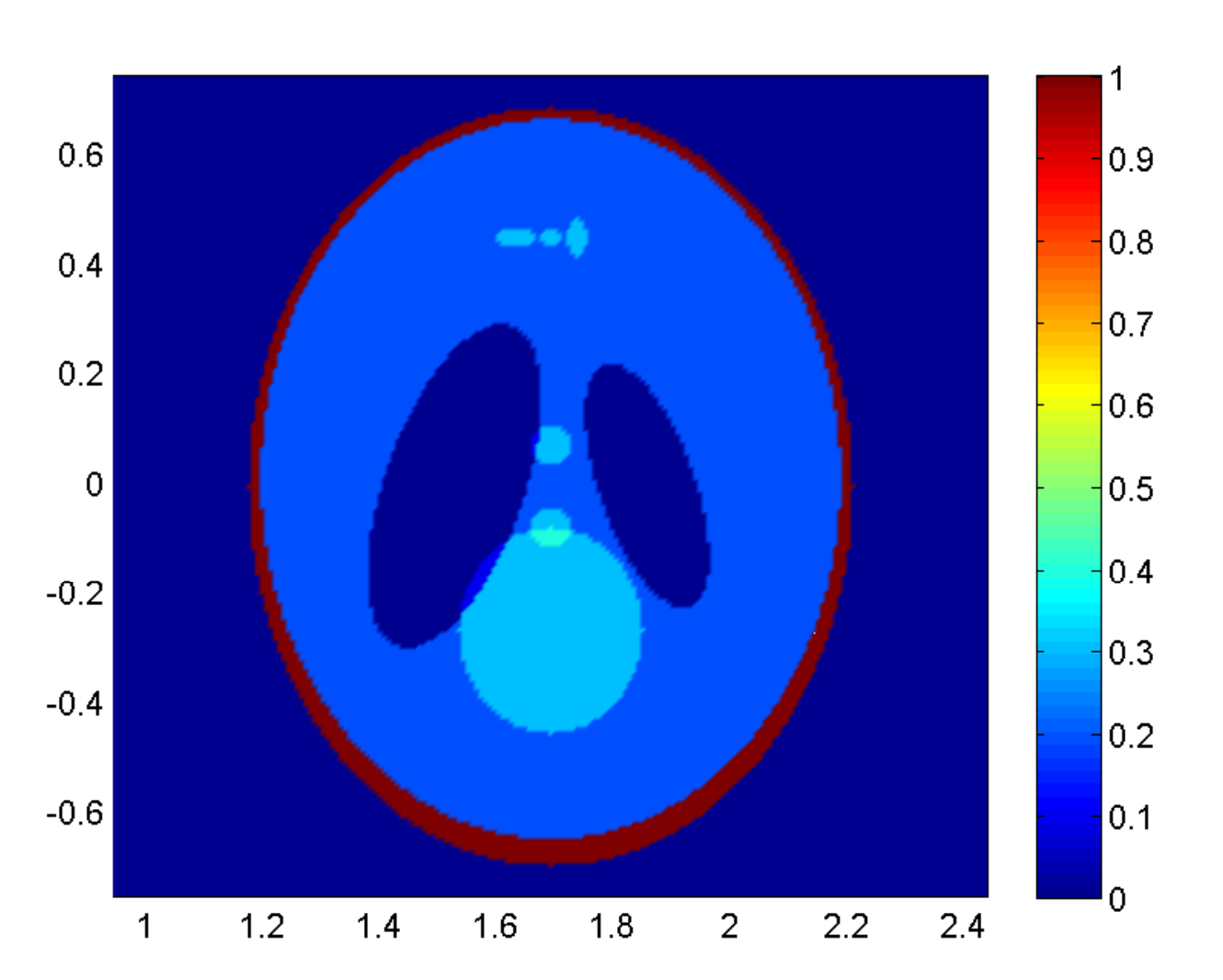}
  \includegraphics[width=0.4\textwidth]{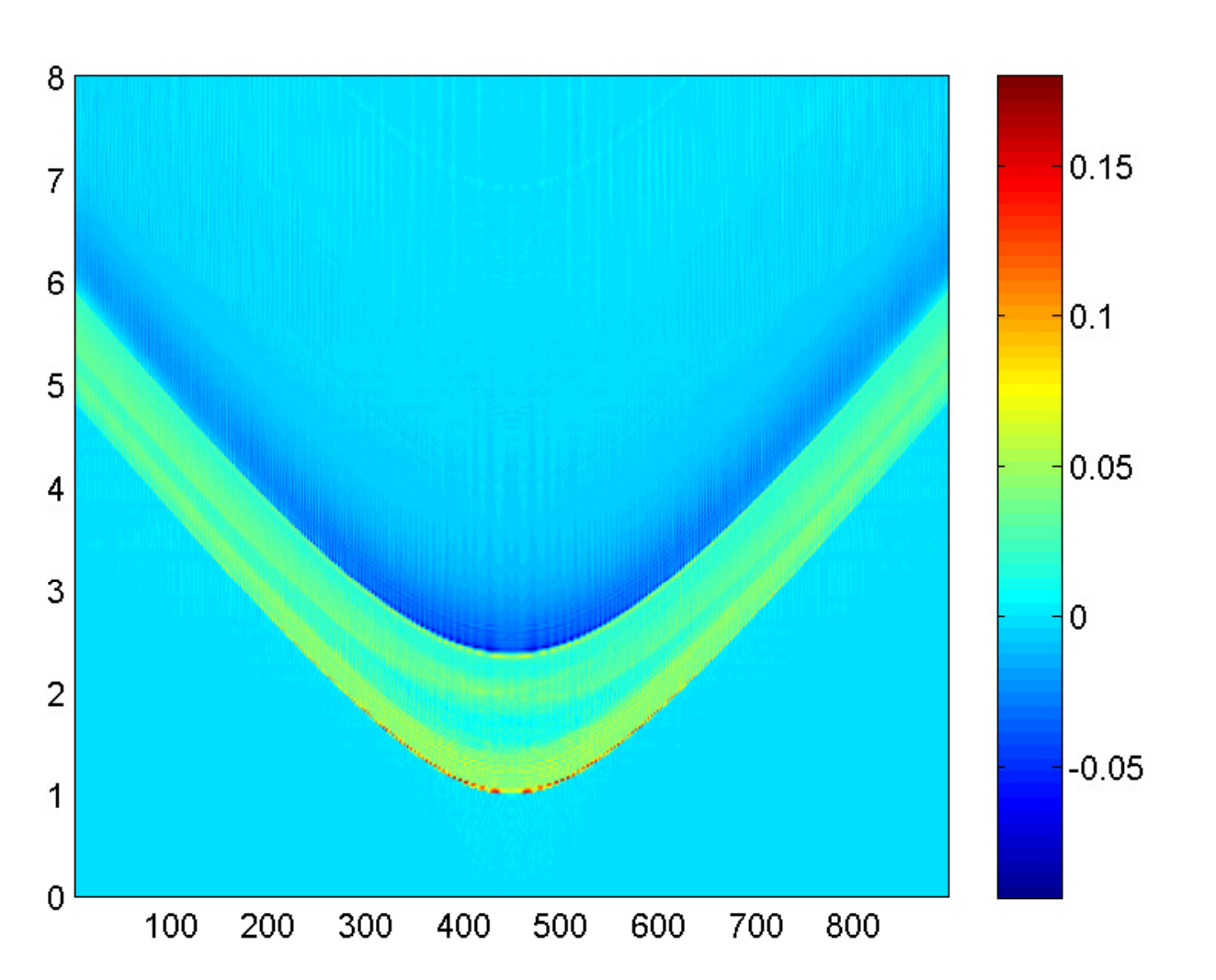}\\
  \includegraphics[width=0.4\textwidth]{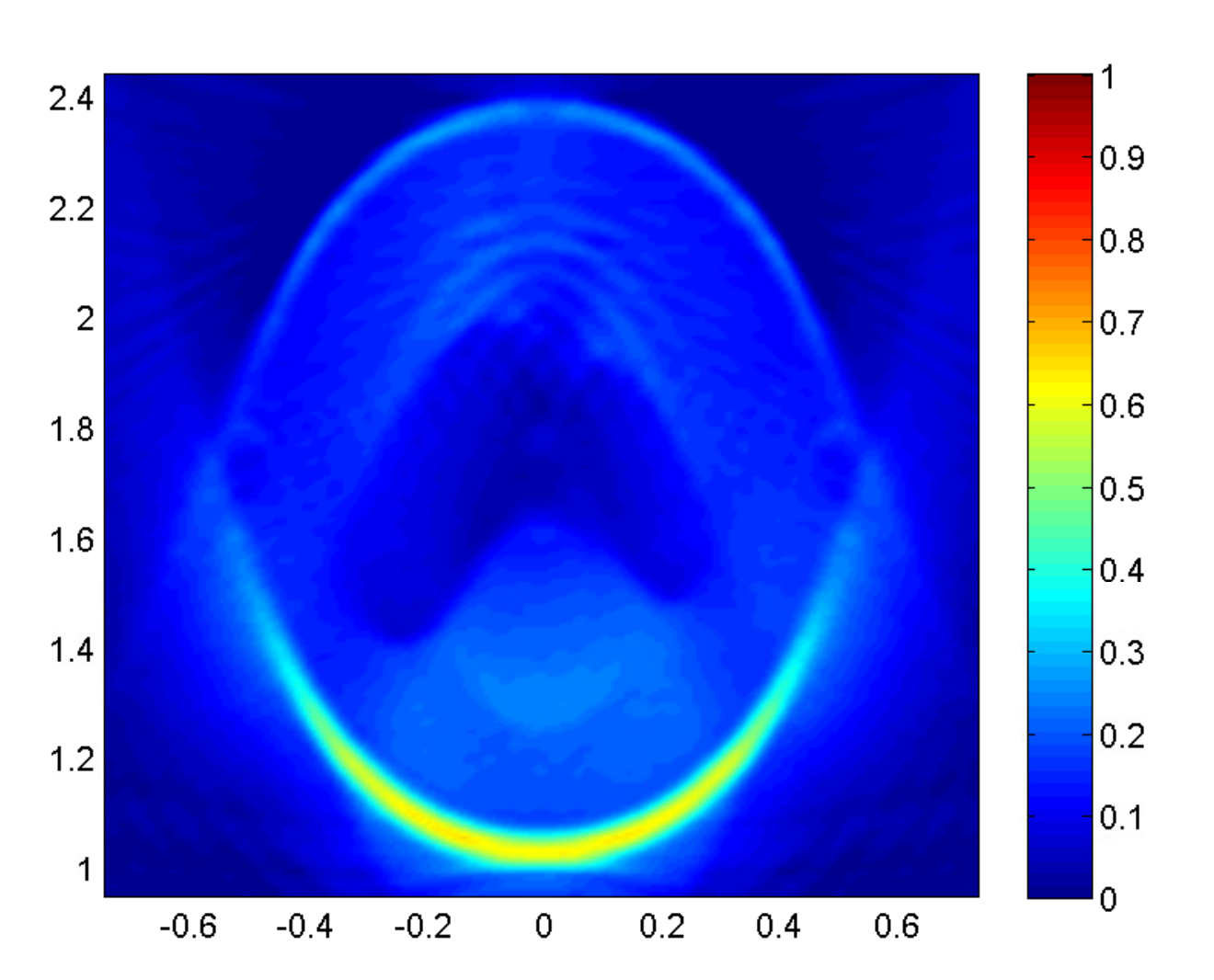}
  \includegraphics[width=0.4\textwidth]{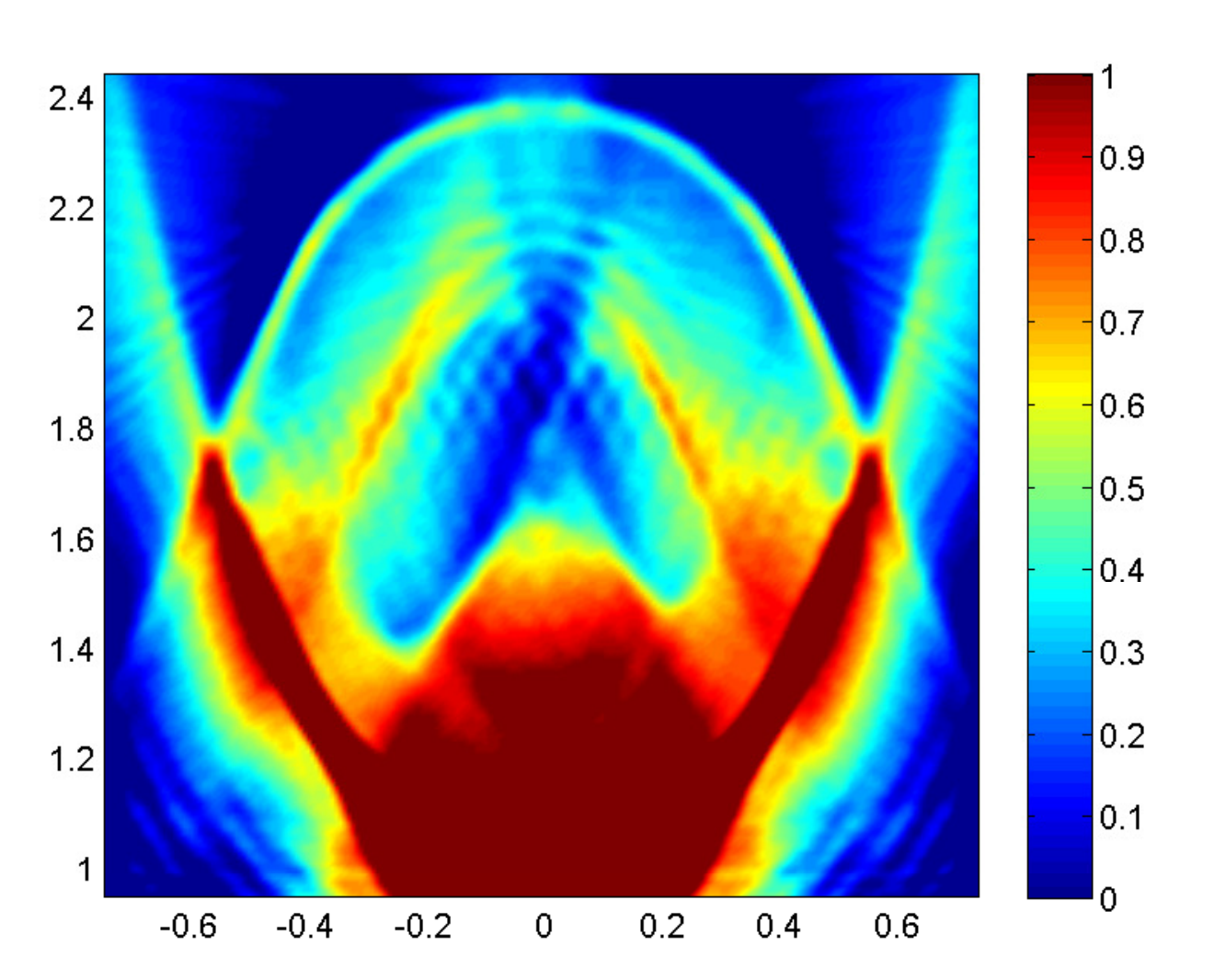}\\
  \includegraphics[width=0.4\textwidth]{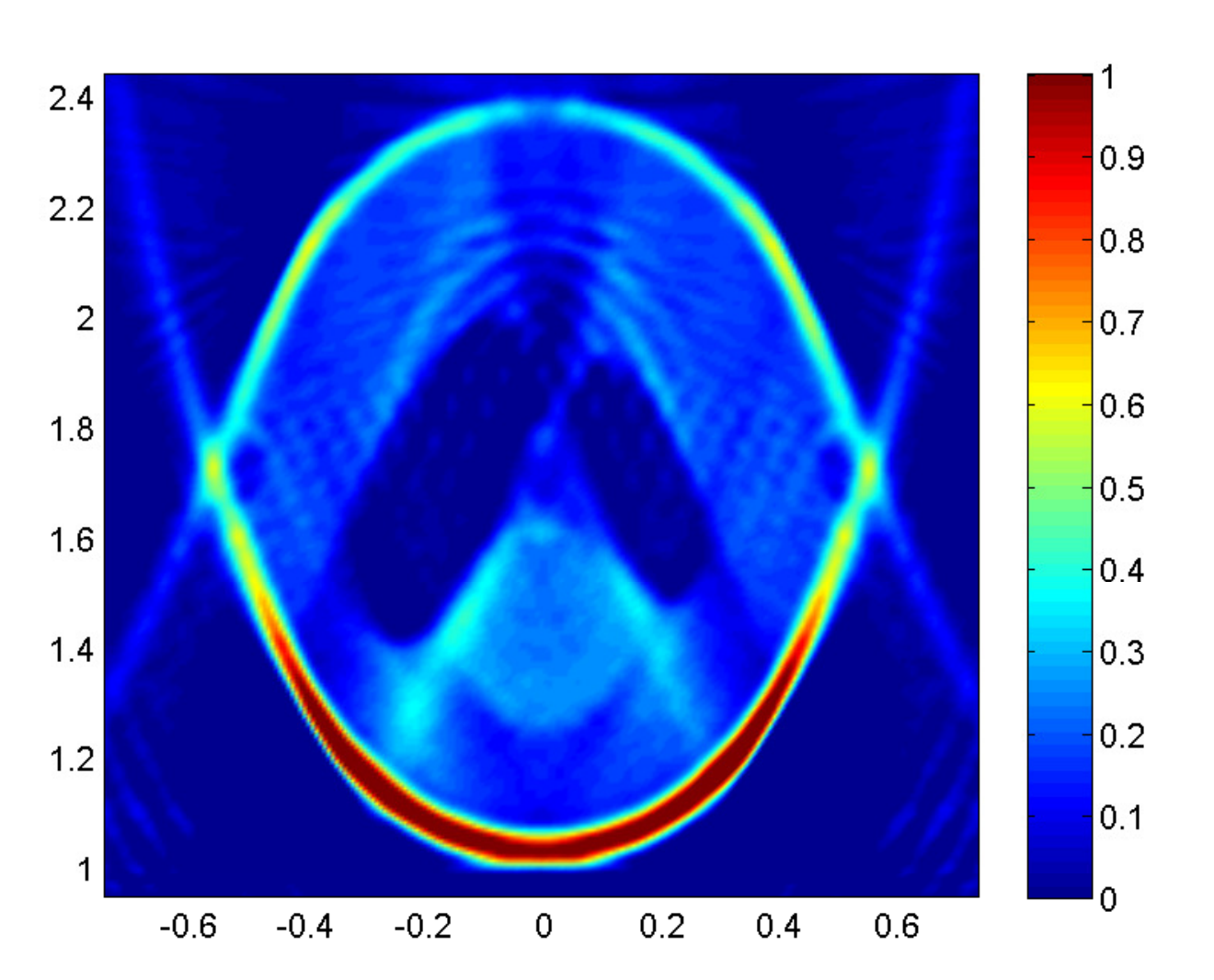}
  \includegraphics[width=0.4\textwidth]{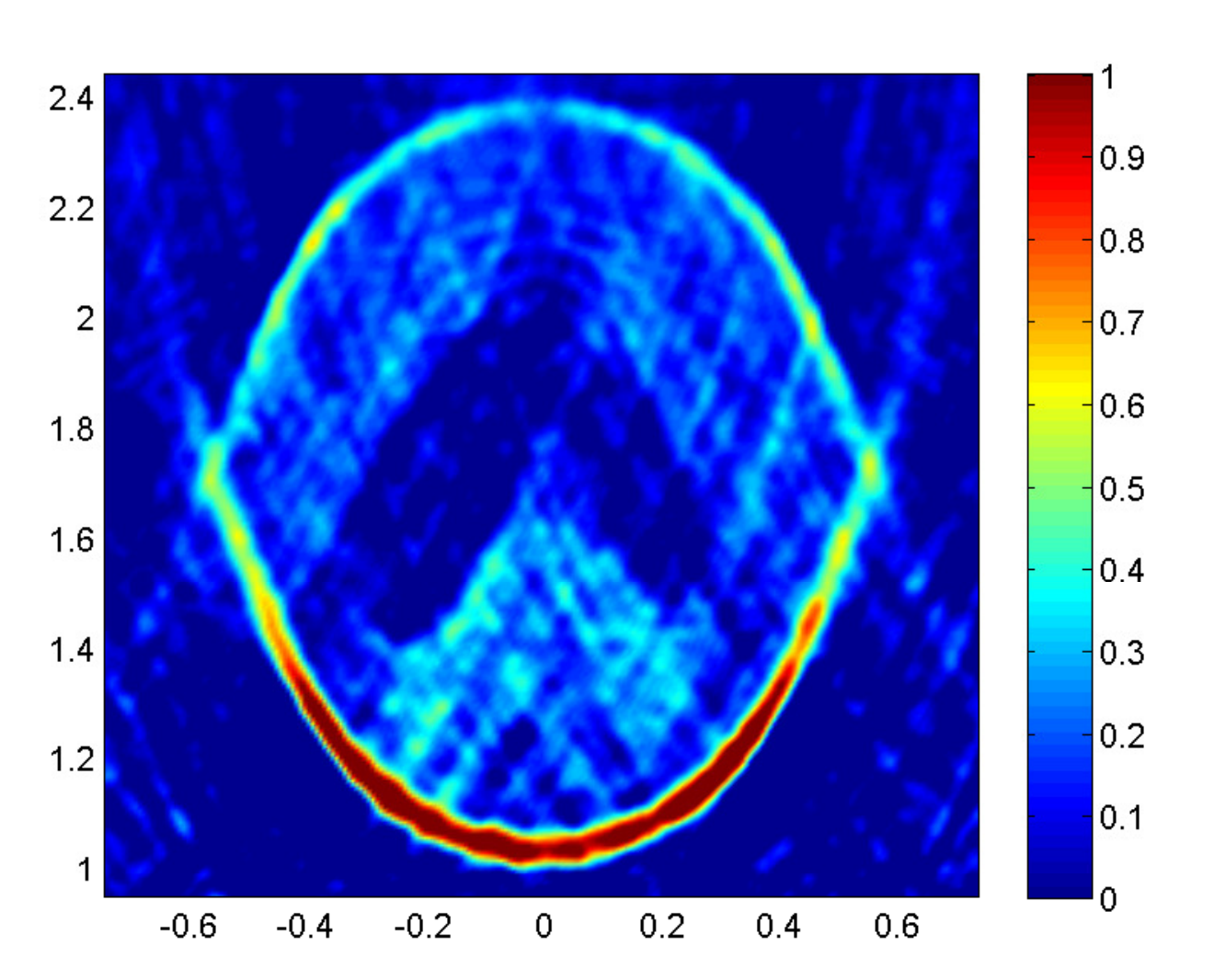}\\
  \includegraphics[width=0.7\textwidth]{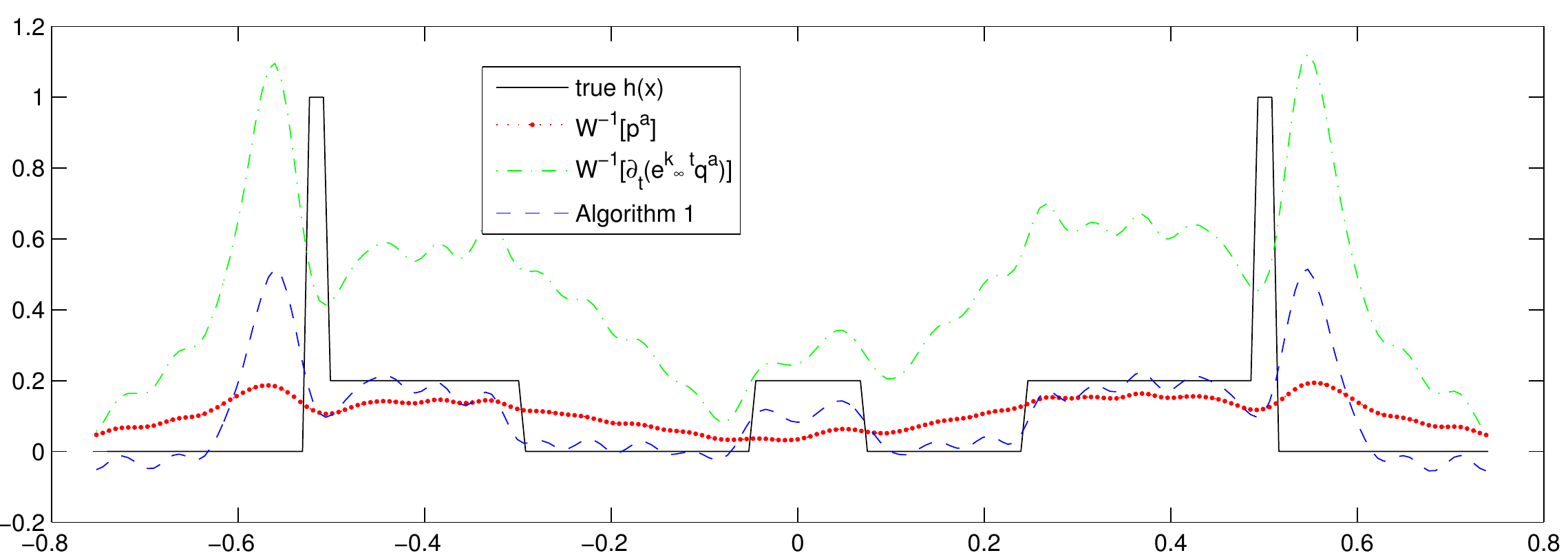}
  \caption{\label{fig:line2}Measurements along a line and NSW model. Top left: Ground truth.
  Top Right: the simulated pressure data $\pa$. Middle: Reconstruction by universal back-projection
  (not taking into account attenuation), by compensated attenuation \autoref{eq:RescaledBp} and by \autoref{pseudo-code} with noise-free and $20\%$ noise.
  Bottom: Cross section through ground truth and reconstructions.}
\end{figure}
\clearpage

\section*{Conclusion}
We have presented explicit reconstruction formulas for photoacoustic imaging in acoustically attenuating media,
which are based on the universal back-projection formula. We have presented a numerical algorithm and showed numerical
reconstructions, which were compared with attenuation compensation techniques. The numerical simulations
show that the new technique can produce better visualization in 2D with a similar numerical complexity.

\section*{Acknowledgment}
The work is support by the ``Doctoral Program Dissipation and Dispersion in Nonlinear PDEs'' (W1245). OS is also supported by the FWF-project ``Interdisciplinary Coupled Physics Imaging'' (FWF P26687).

\section*{References}

\printbibliography[heading=none]

\begin{appendix}
\section{Appendix}
\begin{theorem}
 \label{th:st_phase} \cite[Section 7.7]{Hoe03}
 Let $f,g \in C^\infty(\R;\C)$ satisfying
       \begin{itemize}
        \item $\supp g$ is compact and
        \item $\Im f(\omega) \geq 0$.
       \end{itemize}
       Then there exists a constant $C_1 > 0$ such that for all $l \in \N$ and all $\tau \geq 0$,
       \begin{equation}\label{eq:stationary_phase}
        \begin{aligned}
         ~& \tau^l\left|\intb{\omega} e^{\i \tau f(\omega)} g(\omega) \d \omega\right| \\
         \leq & C_1 \sum_{\alpha=0}^{l} \sup_{\omega \in \R}
         \left|d^\alpha g(\omega)\right|(|f'(\omega)|^2+ \Im f(\omega))^{\alpha/2-l}.
        \end{aligned}
       \end{equation}
\end{theorem}

\begin{lemma}
\label{le:hoermander}
 Let $\kappa$ be an attenuation coefficient (cf. \autoref{deAttCoeff}), then there exists a constant $C_2>0$ such that
 \begin{equation}\label{eq:klb_derivative}
  \abs{\kappa'(\omega)}^2 + \Im \kappa (\omega) \geq C_2\;.
 \end{equation}
\end{lemma}

\begin{proof}
The fourth assumption of \autoref{deAttCoeff} ensures that the maximal speed of propagation $c$ is finite.
Then from \cite[Proposition 2.9]{ElbSchShi16_report}, it follows that the holomorphic extension
$\tilde{\kappa}$ (cf. \autoref{deAttCoeff})of $\kappa$ to the upper half plane can be represented as
\begin{equation}\label{eqNevanlinna}
\tilde\kappa(z) = Az+B+\intb{\nu}\frac{1+z\nu}{\nu-z}\d\sigma(\nu),\quad z\in\H,
\end{equation}
where $A=\frac{1}{c}>0,B\in\R$ are constants, and $\sigma:\R\to\R$ is a monotonically increasing function of bounded variation.

From \cite[Formula (C9)]{Nus72} we know that if $\tilde{\kappa}$ satisfies \autoref{eqNevanlinna}, then
\begin{equation}
 \label{eq:im}
 \nu \to \Im \kappa(\nu) = \pi(1+\nu^2)\sigma'(\nu) \text{ for all } \nu \in \R\;.
\end{equation}
Because, by assumption $\kappa \in C^\infty(\R;\C)$, and because $\sigma:\R \to \R$ is a monotonically increasing
function of bounded variation, we conclude from \autoref{eq:im} that $\sigma'\in C^\infty(\R; \R) \cap L^1(\R; \R)$.
Moreover, for all fixed $z\in\H$, because $\nu -z \neq 0$ for all $\nu \in \R$, we have
$\nu \to \frac{1+\nu^2}{(\nu-z)^2}$ is uniformly bounded and thus the function
\begin{equation*}
 \nu \to \frac{1+\nu^2}{(\nu-z)^2} \sigma'(\nu) \in L^1(\R; \C)\;.
\end{equation*}
Differentiation of \autoref{eqNevanlinna} with respect to $z$ and taking into account that
$\sigma'\in C^\infty(\R; \R) \cap L^1(\R; \R)$
yields
\begin{equation*}
\tilde{\kappa}'(z) = A+\intb{\nu}\frac{(1+\nu^2)}{(\nu-z)^2} \sigma'(\nu) \d\nu,\quad z\in\H.
\end{equation*}
Let now $z=\omega+\i \eta$ and take real parts in the above formula to get
\begin{equation}\label{eq:kappa_derivative1}
\Re\tilde{\kappa}'(\omega+\i \eta) = A+\intb{\nu}(1+\nu^2)\frac{(\nu-\omega)^2-\eta^2}{((\nu-\omega)^2+\eta^2)^2}
\sigma'(\nu)\d\nu.
\end{equation}
We are proving now that there exists a constant $C_r>0$ such that
\begin{equation}\label{eq:kappa_derivative_bound1}
\Re\kappa'(\omega) = \lim_{\eta \to 0^+} \Re \tilde{\kappa}'(\omega + \i \eta) \geq
A-2(1+C_r)\sqrt{(1+\omega^2)\sigma'(\omega)}.
\end{equation}
%

Let
\begin{equation*}
 \begin{aligned}
  \mathcal{A}_+ &:= \set{\nu \in \R: |\nu-\omega|^2 \geq (1+\omega^2)\sigma'(\omega)},\\
  \mathcal{A}_- &:= \set{\nu \in \R: |\nu-\omega|^2 < (1+\omega^2)\sigma'(\omega)},\\
  \mathcal{A}_0 &:= \set{\hat{\nu} \in \R: |\hat{\nu}|^2 < (1+\omega^2)\sigma'(\omega)}\;.
 \end{aligned}
\end{equation*}
For $\omega \in \R$ and $\eta > 0$ let
\begin{equation*}
\nu \to \rho_{\omega,\eta}(\nu):= (1+\nu^2)\frac{(\nu-\omega)^2-\eta^2}{((\nu-\omega)^2+\eta^2)^2}\sigma'(\nu) \text{ for all }
\omega \neq \nu \in \R.
\end{equation*}
The function $\rho_{\omega,\eta}$ can be estimated as follows:
 \begin{equation*}
\rho_{\omega,\eta}(\nu)= \frac{1+\nu^2}{(\nu-\omega)^2+\eta^2}\underbrace{\frac{(\nu-\omega)^2-\eta^2}{(\nu-\omega)^2+\eta^2}}_{\leq 1} \sigma'(\nu)
\end{equation*}
Moreover, since $\nu\in \mathcal{A}_+$, we find that
\begin{equation*}
\begin{aligned}
\frac{1+\nu^2}{(\nu-\omega)^2+\eta^2} \leq \frac{1+\nu^2}{(\nu-\omega)^2}=
\underbrace{\frac{1+\nu^2}{1+(\nu-\omega)^2}}_{\leq 2(1+\omega^2)}
\underbrace{\frac{1+(\nu-\omega)^2}{(\nu-\omega)^2}}_{\leq 1+((1+\omega^2)\sigma'(\omega))^{-1}},\\
\end{aligned}
\end{equation*}
where the inequality $\frac{1+\nu^2}{1+(\nu-\omega)^2}\leq 2(1+\omega^2)$ is a consequence of the algebraic
identity $2(1+\omega^2)(1+(\nu-\omega)^2)-(1+\nu^2)=2\omega^2(\nu-\omega)^2+(\nu-2\omega)^2+1>0$.
Therefore with $C_\omega= 2(1+\omega^2)(1+((1+\omega^2)\sigma'(\omega))^{-1})$ it follows that
\begin{equation*}
 \abs{\rho_{\omega,\eta}(\nu)} \leq C_\omega \sigma'(\nu),
\end{equation*}
and because $\sigma' \in L^1(\R; \R)$ the latter means that the functions $\set{\rho_{\omega,\eta}: \eta > 0}$ are uniformly
dominated by an $L^1(\R; \R)$ function.
Therefore we can apply the dominated convergence theorem and get
\begin{equation}\label{eq:kappa_derivative_bound2}
\lim_{\eta\to0^+} \int\limits_{\mathcal{A}_+} \rho_{\omega,\eta}(\nu) \d\nu
= \int\limits_{\mathcal{A}_+} \rho_{\omega,0}(\nu) \d\nu \geq0.
\end{equation}

To estimate $\int_{\mathcal{A}_-} \rho_{\omega,\eta}(\nu) \d \nu$, we use the Taylor's expansion of
$\omega \to (1+(\hat{\nu}+\omega)^2) \sigma'(\hat{\nu}+\omega)$ and get
\begin{equation}
\label{eq:taylorII}
 (1+(\hat{\nu}+\omega)^2) \sigma'(\hat{\nu}+\omega) = (1+\omega^2) \sigma'(\omega) + \hat{\nu}\left( (1+\omega^2)\sigma''(\omega)+2\omega\sigma'(\omega)\right)+ r(\hat{\nu}),
\end{equation}
with
\begin{equation*}
 \abs{r(\hat{\nu})} \leq C_r \hat{\nu}^2 \text{ for all } \hat{\nu} \in \mathcal{A}_-.
\end{equation*}
Using the substitution $\nu \to\hat{\nu}:=\nu-\omega$ and \autoref{eq:taylorII} we get
\begin{equation}
\label{eq:rho_omega_eta}
\begin{aligned}
~ & \lim_{\eta \to 0^+}\int\limits_{\mathcal{A}_-} \rho_{\omega,\eta}(\nu) \d \nu \\
&= \lim_{\eta \to 0^+} \int\limits_{\mathcal{A}_0} (1+(\hat{\nu}+\omega)^2) \sigma'(\hat{\nu}+\omega)
                       \frac{\hat{\nu}^2-\eta^2} {(\hat{\nu}^2+\eta^2)^2} \d\hat{\nu}\\
&= (1+\omega^2) \sigma'(\omega) \lim_{\eta \to 0^+} \int\limits_{\mathcal{A}_0}
   \frac{\hat{\nu}^2-\eta^2} {(\hat{\nu}^2+\eta^2)^2}\d \hat{\nu}\\
&\quad+
   \left( (1+\omega^2)\sigma''(\omega)+2\omega\sigma'(\omega)\right) \lim_{\eta \to 0^+}
   \int\limits_{\mathcal{A}_0}\hat{\nu}\frac{\hat{\nu}^2-\eta^2} {(\hat{\nu}^2+\eta^2)^2}\d \hat{\nu}\\
& \quad + \lim_{\eta \to 0^+} \int\limits_{\mathcal{A}_0} r(\hat{\nu}) \frac{\hat{\nu}^2-\eta^2} {(\hat{\nu}^2+\eta^2)^2}\d \hat{\nu}
\end{aligned}
\end{equation}
Plugging in the integral formulas
\begin{equation*}
\int\limits_{-a}^{a}\frac{\nu^2-\eta^2}{(\nu^2+\eta^2)^2}\d\nu=\frac{-2a}{a^2+\eta^2} \text{ and }
\int\limits_{-a}^{a}\nu\frac{\nu^2-\eta^2}{(\nu^2+\eta^2)^2}\d\nu=0
\end{equation*}
into \autoref{eq:rho_omega_eta} we get for the first term with $a = \sqrt{(1+\omega^2)\sigma'(\omega)}$
\begin{equation}\label{eq:NevanSmallLowOrder}
\begin{aligned}
& (1+\omega^2) \sigma'(\omega) \lim_{\eta\to0^+} \int\limits_{\mathcal{A}_0} \frac{\hat{\nu}^2-\eta^2} {(\hat{\nu}^2+\eta^2)^2}\d\hat{\nu}\\
& = (1+\omega^2)\sigma'(\omega)\lim_{\eta\to0^+}\frac{-2\sqrt{(1+\omega^2)\sigma'(\omega)}}{(1+\omega^2)\sigma'(\omega)+\eta^2}=-2\sqrt{(1+\omega^2)\sigma'(\omega)},
\end{aligned}
\end{equation}
and the second integral term in \autoref{eq:rho_omega_eta} is vanishing, and for the third term we get
\begin{equation}\label{eq:NevanSmallHighOrder}
\left|\int\limits_{\mathcal{A}_0} r(\hat{\nu}) \frac{\hat{\nu}^2-\eta^2}{(\hat{\nu}^2+\eta^2)^2}\d\hat{\nu}\right|
\leq C_r \int\limits_{\mathcal{A}_0} 1 \d\hat{\nu} \leq
2C_r\sqrt{(1+\omega^2)\sigma'(\omega)}.
\end{equation}
Using the estimates \autoref{eq:NevanSmallLowOrder} and \autoref{eq:NevanSmallHighOrder} in \autoref{eq:rho_omega_eta} we get
\begin{equation}\label{eq:kappa_derivative_bound3}
 \lim_{\eta \to 0^+}\int\limits_{\mathcal{A}_-}  \rho_{\omega,\eta}(\nu) \d \nu \geq -2(1+C_r) \sqrt{(1+\omega^2)\sigma'(\omega)}.
\end{equation}
Considering the integral \autoref{eq:kappa_derivative1} as the sum of the two integrals over $\mathcal{A}_{\pm}$
and using the estimates \autoref{eq:kappa_derivative_bound2} and \autoref{eq:kappa_derivative_bound3} we get
\begin{equation*}
\lim_{\eta\to0^+}\Re\tilde{\kappa}'(\omega+\i \eta)\geq A-2(1+C_r)\sqrt{(1+\omega^2)\sigma'(\omega)}.
\end{equation*}
Therefore, $|\Re\kappa'(\omega)|\geq\max(0,A-2(1+C_r)\sqrt{(1+\omega^2)\sigma'(\omega)})$ and together with
\autoref{eq:im} it follows that
\begin{equation*}
\begin{aligned}
~& |\kappa'(\omega)|^2+\Im\kappa(\omega)\\
\geq &\pi (1+\omega^2)\sigma'(\omega)+|\Re \kappa'(\omega)|^2\\
\geq & \pi(1+\omega^2)\sigma'(\omega)+\left( \max \set{0,A-2(1+C_r)\sqrt{(1+\omega^2)\sigma'(\omega)}} \right)^2\\
\geq & \underbrace{\frac{\pi A^2}{4(1+C_r)^2+\pi}}_{:=C_2},
\end{aligned}
\end{equation*}
where in the last inequality we estimated the minimum of the quadratic function
$\rho \to A^2 - 4A (1+C_r) \rho + (4(1+C_r)^2+\pi)\rho^2$.
\end{proof}

\end{appendix}
\end{document}